\documentclass[a4paper, 12pt]{amsart}

\title{Some remarks on the geometry of the Standard Map}
\author{Katie Bloor}
\address{Mathematics Department, Imperial College London, SW7 2AZ}
\email{katie.bloor98@imperial.ac.uk}

\author{Stefano Luzzatto}
\address{Mathematics Department, Imperial College London, SW7 2AZ}
\email{Stefano.Luzzatto@imperial.ac.cuk}
\urladdr{http://www.ma.ic.ac.uk/~luzzatto}

\thanks{We would like to thank Oliver Butterley for several
discussions and for help in drawing Figure \ref{fig:foliations}.
Figure \ref{ksmall} was drawn using the software \texttt{StdMap}
by J. D. Meiss. This research was partly supported by EPSRC grant
T0969901.
}
\date{\today}

\usepackage{amsmath,amssymb,amsthm,
amsfonts,latexsym}
\usepackage[pdftex]{graphicx}
\usepackage{wrapfig}
\usepackage{a4wide}
\usepackage{fancybox}

\newtheorem{theorem}{Theorem}
\newtheorem{lemma}{Lemma}
\newtheorem{proposition}{Proposition}

\newtheorem{corollary}{Corollary}

\theoremstyle{remark}
\newtheorem{remark}{Remark}

\theoremstyle{definition}
\newtheorem{question}{Question}

\def\cprime{$'$}
\def\polhk#1{\setbox0=\hbox{#1}{\ooalign{\hidewidth
    \lower1.5ex\hbox{`}\hidewidth\crcr\unhbox0}}}

\newcommand{\Tt}{\mathbb T^2}

\newcommand{\p}{2\pi k}

\newcommand{\pfi}[2]{\ensuremath{\partial_{#1}f_{#2}}}

\begin{document}
\begin{abstract}
 We define and compute \emph{hyperbolic coordinates} and associated
 foliations  which provide a new way to
describe the geometry of the standard map. We also identify a
uniformly hyperbolic region and a complementary ``critical'' region
containing a smooth curve of tangencies between certain canonical
``stable'' foliations. 
 \end{abstract}   

\maketitle


\section{Introduction and informal statement of results}
\subsection{The standard map}

The Standard Map family is a one-parameter family of maps 
\(
f_k : \Tt\to\Tt
\) on the 2-torus, with a parameter  \( k\in\mathbb R^{+} \),
 given  by 
\[
f_{k}\begin{pmatrix}x\\y\end{pmatrix} =
\begin{pmatrix}x + k\sin(2\pi y)  \\x
+ y + k\sin(2\pi y) \end{pmatrix} 
\mod 1. 
\]
This family was  introduced independently 
by Taylor and by Chirikov \cite{Chi79} in the late 60's and is
related to a variety of problems in many areas of physics, 
see for example \cite{Mei92, Sin94, Mac95} for extensive discussions. 
For \( k=0 \) the map reduces to \( f_{0}(x,y) = (x+y, y) \)
which is completely integrable (see \cite{Mei92} for full
definitions) and \( \mathbb T^{2} \) is foliated by invariant closed
curves.
\begin{wrapfigure}[15]{r}{6cm}
    \begin{center}
 \includegraphics[height=5.5cm]{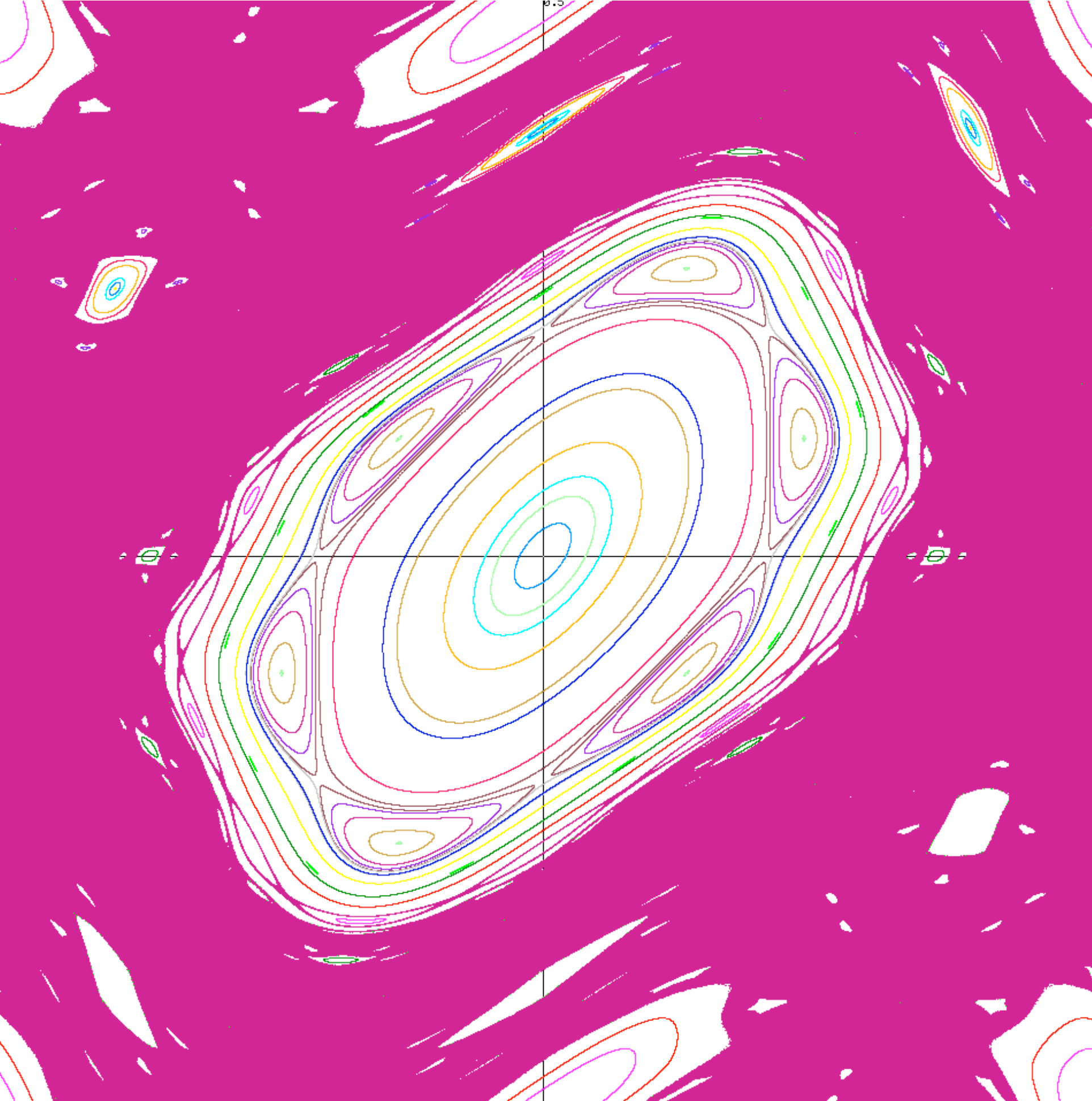}   
\end{center}
\caption{Elliptic islands in a chaotic sea} 
 \label{ksmall}
 \end{wrapfigure}   
 As \( k \) increases these invariant curves break up; the 
 region of parameter values with \( k \) small 
 has provided one of the main examples for
 the application of the KAM theory and its developments such as the
 studies on the splitting of separatrices \cite{Gel99}, 
 formation and properties of Cantori \cite{MacMeiPer84} etc.,
 see \cite{Lla01} for a comprehensive review of the theory and
 and references. 
 Computer pictures show complicated structures in which regions of
 invariant circles co-exist with apparently more complicated ``chaotic''
 dynamics, see Figure \ref{ksmall}. 
 As \( k \) increases further these ``elliptic islands'' shrink until
they become essentially invisible and the ``chaotic'' region 
seemingly takes
over more and more of the phase space.  
A rigorous description of the 
structure of the dynamics for large \( k \) appears to be an extremely
difficult problem. 

\newpage
A basic but still open
question is the following.

\begin{question}
Are there any values of the parameter \( k \) for which \( f_{k} \)
has positive metric entropy ? Equivalently, are there any values of
the parameter \( k \) for which \( f_{k} \) has an ergodic component
of positive measure on which it is (nonuniformly) hyperbolic ?
\end{question}

Numerical evidence suggests a positive answer to this question for
large values of \( k \) and  it 
seems  reasonable to expect the existence of
several, possibly infinitely many, ergodic components of positive
measure \cite{Pus86}  perhaps
co-existing with elliptic islands (examples of systems in which these
two phenomena are \emph{known} to co-exist are not many but have 
been obtained, e.g. in \cite{Woj81}). However, so far the only really
rigorous results available point in the opposite direction: there is a
dense set of values of \( k \) for which the standard map exhibits
many elliptic points \cite{Dua94}.

It is generally accepted that any progress on this question will
require some detailed concrete analysis of the geometry of the standard
map, although it is not at all clear from what point of view this
geometry should be studied. Some papers have appeared, 
we just mention 
\cite{ChrPol96, ChrPol97} in which generating partitions are
constructed leading to a symbolic description of the dynamics, and 
 a series of papers by Lazutkin
\emph{et al} \cite{GelLaz01,GioLaz00,
GioLazSim97,Laz03,Laz00,Laz99} suggesting a possible strategy, 
however a complete 
implementation of this argument is, to our knowledge, not yet
available.  
We mention also \cite{Pus00} in which a version of the
entropy conjecture is proved for some ``random approximations'' of the
standard map. Other families of maps have also been studied which
exhibit analogous phenomena as the standard map, see for example the
\emph{fundamental map} introduced in \cite{BroChu96}. 

\subsection{Hyperbolic coordinates and approximate critical points}

The aim of this paper is to carry out an analysis of certain
geometrical features of the standard map which, to our knowledge, have
not been investigated before. This analysis is based on the notion of
\emph{hyperbolic coordinates} which essentially consists of carrying
out a singular value decomposition of matrix of the derivative \( Df_{k}(x) \)
at each point \( x \), and thus decomposing the tangent space into
orthogonal subspaces which are most contracted and most expanded by \( 
Df_{k}(x) \). We shall show (Theorem 1) that this decomposition is
well defined at every point \( x \) both for \( f_{k} \) and well as
for its inverse \( f_{k}^{-1} \), that it depends smoothly on \( x \)
and that it defines foliations of the torus \( \mathbb T^{2} \), see
Figure \ref{fig:foliations}. We obtain explicit formulas which allow
us to describe the geometry of these foliations in some detail. 

The singular value decomposition of a two by two matrix is quite different 
from the eigenspace decomposition, even in the case of hyperbolic
diagonalizable matrices. A first key observation is that the singular 
value decomposition always yields orthogonal subspaces, whereas this
is clearly not necessarily the case for the eigenspace decomposition. 
A second, non trivial, observation is that in the case of a hyperbolic 
matrix \( A \), the contracting subspaces \( e^{(n)} \) of the
singular value decomposition of the iterates \( A^{n} \) of the matrix,
actually converge (exponentially fast) to the contracting eigenspace.
It is of course not true that the corresponding expanding subspaces \( 
f^{(n)} \) of the singular value decomposition converge to the
expanding eigenspace, since \( f^{(n)} \) is always orthogonal to \(
e^{(n)} \). However, we can consider the inverse \( A^{-1} \) of the
matrix \( A \), for which the contracting and expanding eigenspaces
are swapped, and obtain the expanding eigenspace of \( A \) which is
exactly the contracting eigenspace of \( A^{-1} \) as a limit
of most contracted directions \( e^{(-n)} \) for the singular value
decomposition for \( A^{(-n)} \).

The notion and properties of hyperbolic coordinates have
been systematically developed and applied in \cite{HolLuz06} 
for the proof of a version of the Stable Manifold Theorem under quite 
weak hyperbolicity assumptions and were originally inspired by certain 
arguments in \cite{BenCar91} and \cite{MorVia93} on the geometry of
H\'enon and H\'enon-like maps. 
Indeed, the properties of the convergence of the 
contracting subspaces apply not only to the iteration of a fixed
matrix but also to the composition of different matrices, as occurs
naturally, for example, 
when considering derivatives \( Df^{n}(x) \) of higher iterates \(
f^{n}(x) \) of some map. If the orbit of the not necessarily periodic 
point \( x \) has some
hyperbolic decomposition (which generalizes the eigenspace
decomposition in the case of a hyperbolic fixed or periodic point),
then the contracting directions \( e^{(n)} \) converge to the
contracting subspace of this decomposition and the contracting
directions \( e^{(-n)} \) for the inverse map converge to the
contracting subspace of the decomposition for the inverse map is
exactly the expanding direction for the original map. 
\begin{figure}[h]
    \begin{center}
\includegraphics[width=9cm]{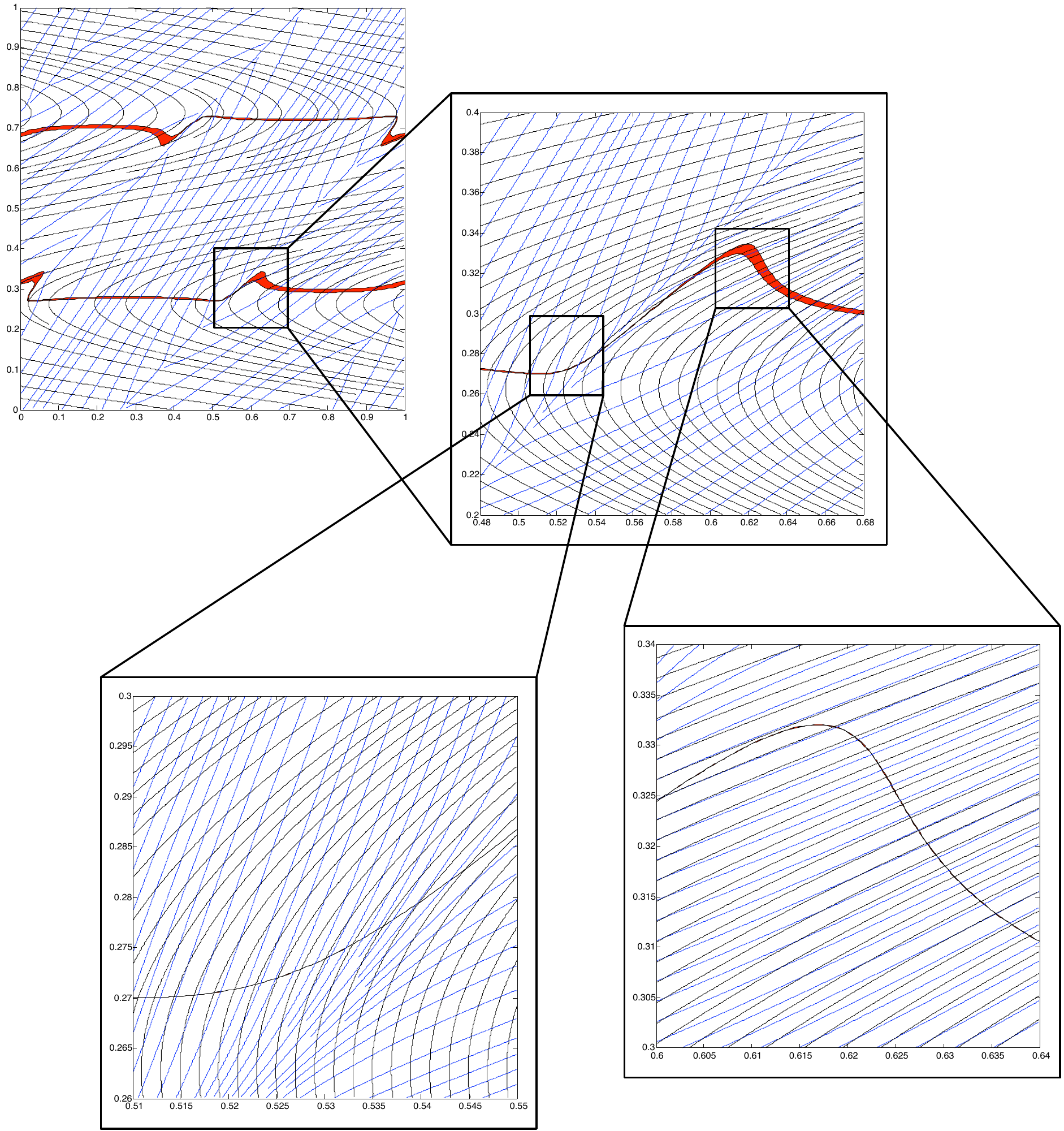}
\end{center}
\caption{Tangencies between \( \mathcal E^{(1)} \) and \(
\mathcal E^{(-1)} \).}
\label{tangencies}
\end{figure}

This approach however really comes into its own when considering
points which do not have a hyperbolic decomposition but are points of 
tangencies between stable and unstable manifolds. Within the framework
of hyperbolic coordinates these can be described as points for which
the contracting directions \( e^{(n)} \) in forward time 
and the contracting directions \( e^{(-n)} \) in backward time both
converge but instead of converging to different subspaces, they 
\emph{converge to the same subspace}. Indeed, it follows from the
properties of stable and unstable manifolds that tangent directions
are contracting in forward and backward time respectively. 
Motivated by these observations we investigate the locus of tangencies
between the the most contracting direction \( e^{(1)} \) for the
standard map \( f_{k} \) and the most contracting directions \( e^{(-1)} \)
for its inverse \( f_{k}^{-1} \). We shall show (Theorem 2) that these tangencies 
occur along two smooth curves, see Figure \ref{tangencies}, 
for which we shall derive explicit
expressions which allow us to describe their geometry in considerable 
detail.
These curves are contained in two quite narrow horizontal strips of
height of the order of \( 1/k \). To complement this description we
also show (Theorem 3) that \( f_{k} \) is uniformly hyperbolic outside
these strips.

\subsection{General remarks}

The study carried out in this paper does not require any very
sophisticated arguments or calculations and, in itself, does not yield
any spectacular conclusions. It should rather be thought of as a
preliminary study of the geometry of the standard map from a novel
point of view (though some related foliations to 
those discussed here appear as part of the 
arguments in \cite{Dua94}) 
and with the aim of suggesting certain techniques and
directions which might lead to more serious and interesting
conclusions. 

A full resolution of the so-called ``entropy
conjecture'', i.e. the existence of parameter values for which the map
has positive metric entropy is perhaps still too ambitious a goal, but
there are some natural partial results in that direction which seem
more feasible. The first one concerns the geometry of such parameters.
If the solution to the entropy conjecture is to be achieved through
geometric arguments, it would make sense to understand what
characteristics the geometry of such a map would have. The state of
affairs at the moment is more like trying to prove that a certain
phenomenon occurs for some parameter values, 
but having no idea what we are really looking for or what the geometry
of the map would look like at such parameter values. 
We therefore formulate the following problem.

\begin{question}
Construct a geometric model of the
standard map with positive entropy. 
\end{question}
    
Models of this kind have been proposed by simplifying the map in
certain key aspects. What we mean here however is to describe a model 
for the occurrence of positive entropy within some member of the actual
Standard Map family \( f_{k} \) as defined above.  
Of course such a model would rely on highly non-trivial
assumptions about the dynamics, but understanding such a model
could lead to an improved understanding of the actual dynamics of the 
standard map and possibly to a strategy for verifying the assumptions 
for some parameter values. 

One strategy for building such a model would be to start from the
geometrical picture described in this paper: two smooth curves of
critical points inside two very thin strips and uniform hyperbolicity 
outside these strips. Considering both forward and backward iterates 
of the critical strips it should be possible to construct higher order
hyperbolic coordinates and thus higher order critical curves, at least 
in certain parts of the critical strip. Imposing a bounded recurrence 
condition to the critical strip it may be possible to set up an
inductive argument yielding a positive measure set on which the 
map is nonuniformly hyperbolic and thus has positive metric entropy.
Using this strategy it might be possible also to address the following
problem which to our knowledge has not been stated before but which
would seem to be of independent interest. 

\begin{question}
    Show that for all sufficiently large values of the parameter \( k
    \), any ergodic component of positive measure is (nonuniformly)
    hyperbolic.
    \end{question}

 An affirmative answer to this question would reduce the entropy
 conjecture to showing that for arbitrarily large values of \( k \)
 there are ergodic components of positive measure, though it is not
 clear that this would simplify the problem.    

 \section{Formal statement of results}
 
 The precise statements of our three theorems require a significant
 amount of notation. We shall therefore state them in three separate
 section, preceded by the relevant definitions. 
 
\subsection{Hyperbolic Coordinates}\label{hypcoord}
We start with some general definitions which apply to any 
surface diffeomorphism \( f: M \to M \). 
For \( z\in M \) and \( n\geq 1 \)
let
\[
F_{n}(z) = \|(Df^{n}_{z})\| = 
\max_{\|v\|=1}\{\|Df^{n}_{z}(v)\|\}, 
\]
and 
\[
E_{n}(z) =
\|(Df^{n}_{z})^{-1}\|^{-1} = 
\|Df^{-n}_{f^{n}(z)}\|^{-1}
=\min_{\|v\|=1}\{\|Df^{n}_{z}(v)\|\}.
\]  
These quantities have a  simple geometric interpretation: 
 since \( Df_{z}^{n}: T_{z}M \to 
 T_{f(z)}M \) is a linear map, it sends circles to ellipses, 
 see Figure  \ref{HypCoord},
 \begin{figure}[h]
 \includegraphics[width=8cm]{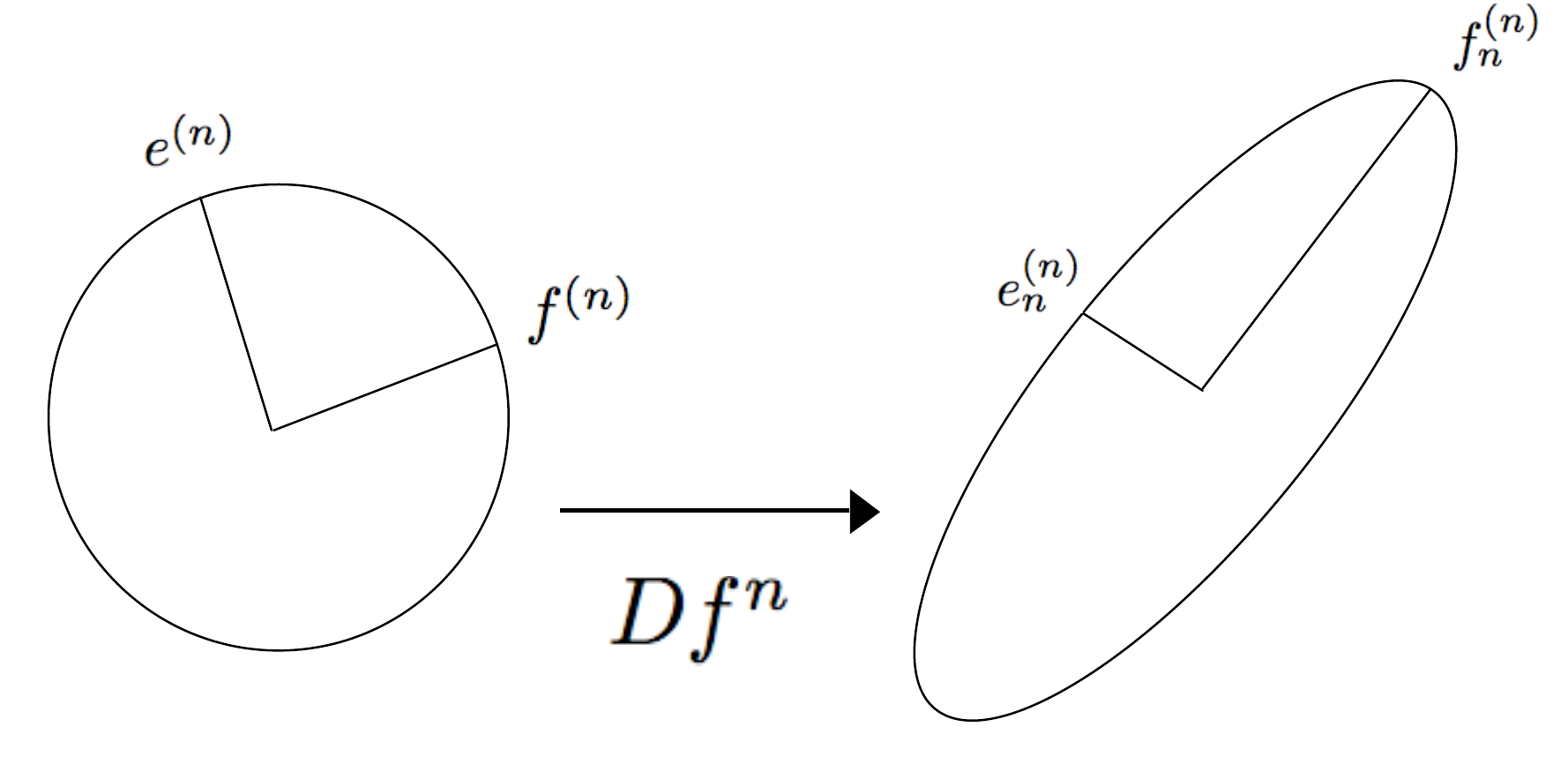}
 \caption{Hyperbolic coordinates of order \( n \)}
 \label{HypCoord}
 \end{figure}
 then
  \( F_{n}(z)\) is precisely half the length of 
 major axis of this ellipse and \( E_{n}(z) \) is 
 precisely half the length of the minor axis of this ellipse. Then let 
 \(
 H_{n}(z) = E_{n}(z)/F_{n}(z). 
 \)
 Clearly we always have \( H_{n}(z) \leq 1 \). 
 \( H_{n}(z)=1 \), or \( E_{n}(z)=F_{n}(z) \),  means that the
 image of the unit circle under \( Df^{n} \) is itself a circle and \( 
 Df^{n} \) is therefore a conformal linear map. On the other hand
 \( H_{n}(z) < 1 \)
is in some sense a (very weak) 
  \emph{hyperbolicity condition} which 
 implies that  the image of the unit circle is 
strictly an ellipse and that therefore distinct \emph{orthogonal} 
unit vectors 
\[ 
e^{(n)}(z) \quad\text{ and }\quad  f^{(n)}(z) 
\]
are defined  in the 
\emph{most contracted} and \emph{most expanded} directions 
respectively by  \( Df^{n}_{z} \).
We can use \( e^{(n)}(z) \) and \( 
f^{(n)}(z) \) as basis vectors for the tangent space \( T_{z}M \)
thus can think of them as defining a new system of coordinates which
we call \emph{hyperbolic coordinates}. 
\begin{wrapfigure}{r}{4cm}
    \includegraphics[width=3cm]{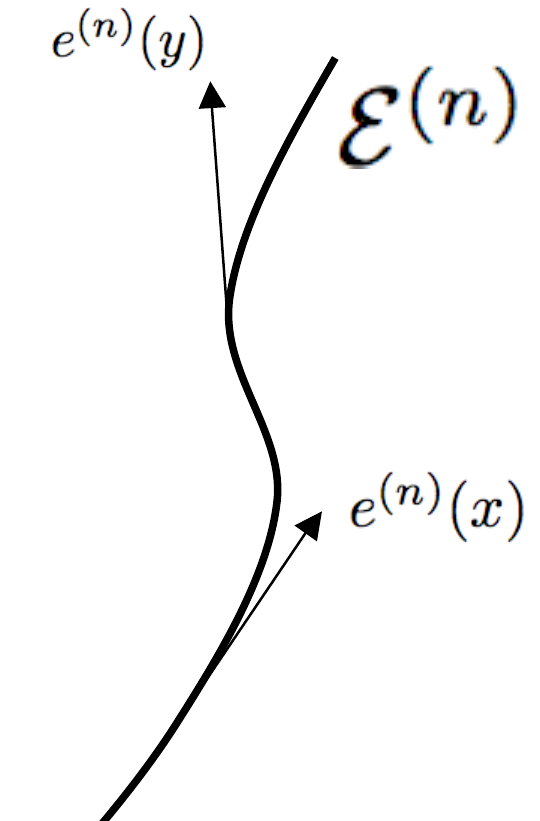}
 \end{wrapfigure} 
The most contracted and most expanded
directions, in the open regions in which they
are defined, determine direction fields whose regularity 
is exactly that of the partial derivatives of \( f \). Thus, under
mild regularity conditions, these open
regions  admit local integral curves, or foliations, 
which we denote by 
\[ 
\mathcal E^{(n)} \quad \text{ and } \quad \mathcal F^{(n)}.
\]
We can think of the curves of these foliations as \emph{finite time}
stable and unstable manifolds.  Indeed, for finite time these curves 
are more relevant than
the real stable and unstable manifolds since they are in some sense
the most contracted and most expanded curves under  \( f^{n} \). 

Since \( f \) is invertible, all the notions and definitions given
above can be applied to \( f^{-1} \). We shall  talk about
hyperbolic coordinates in \emph{forward time} when referring to the
application of these ideas to \( f \) and \emph{backward time} when
referring to the application of these ideas to \( f^{-1} \). 

The first main result of this paper is the detailed description of the
geometry of the foliations \( \mathcal E^{(\pm 1)} \) and \( \mathcal
F^{(\pm 1)} \)  for the standard map and its inverse. 

\begin{theorem}
Contracting foliations \( \mathcal E^{1}, \mathcal E^{(-1)}  \) and
expanding foliations \( \mathcal F^{(1)}, \mathcal F^{(-1)} \) 
are defined everywhere and have the geometry 
illustrated in the  Figure \ref{fig:foliations} and discussed below. 
 \end{theorem}   

 \begin{figure}[h]
     \begin{center}
\includegraphics[width=8cm, height=8cm]{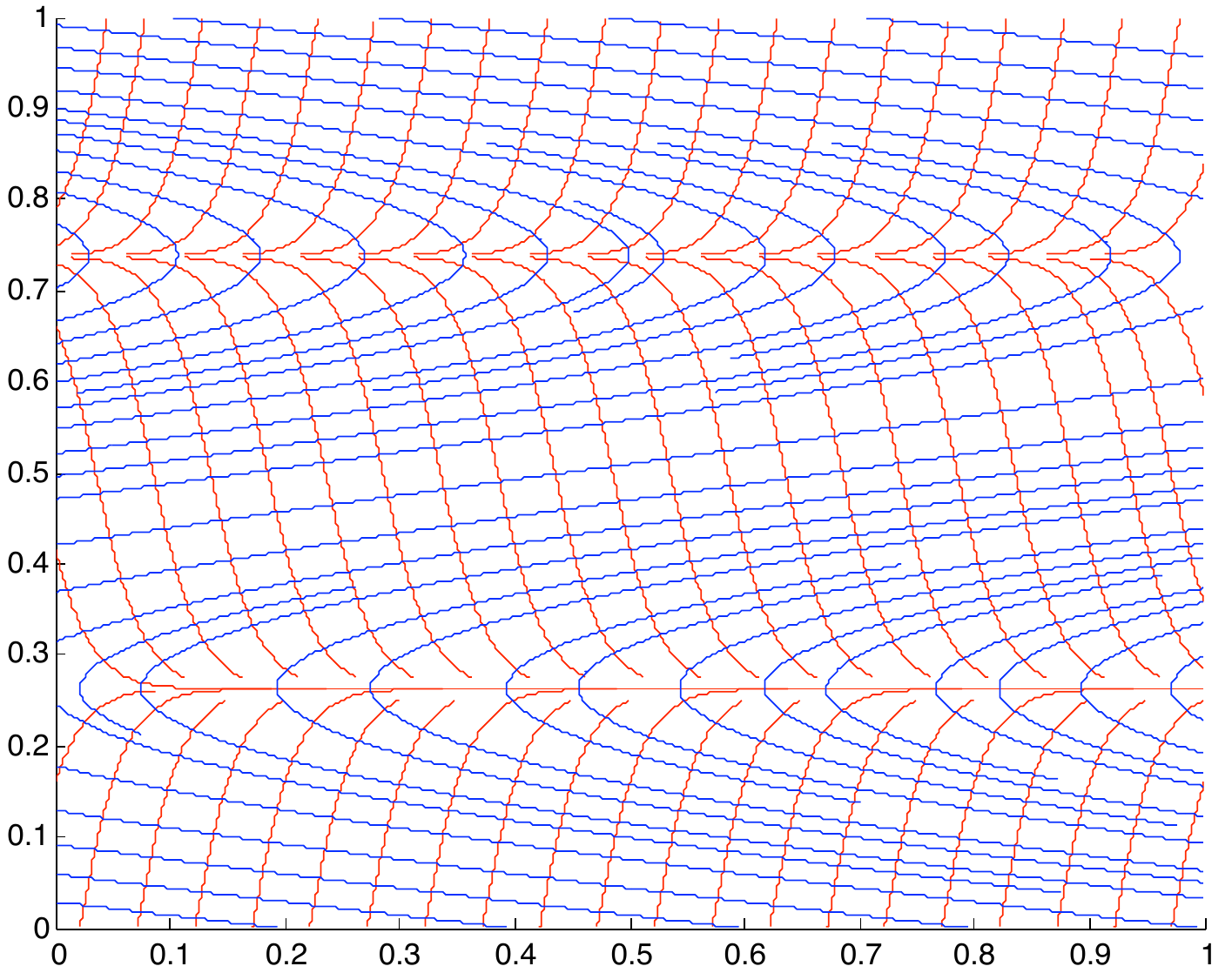}
\includegraphics[width=8cm, height=8cm]{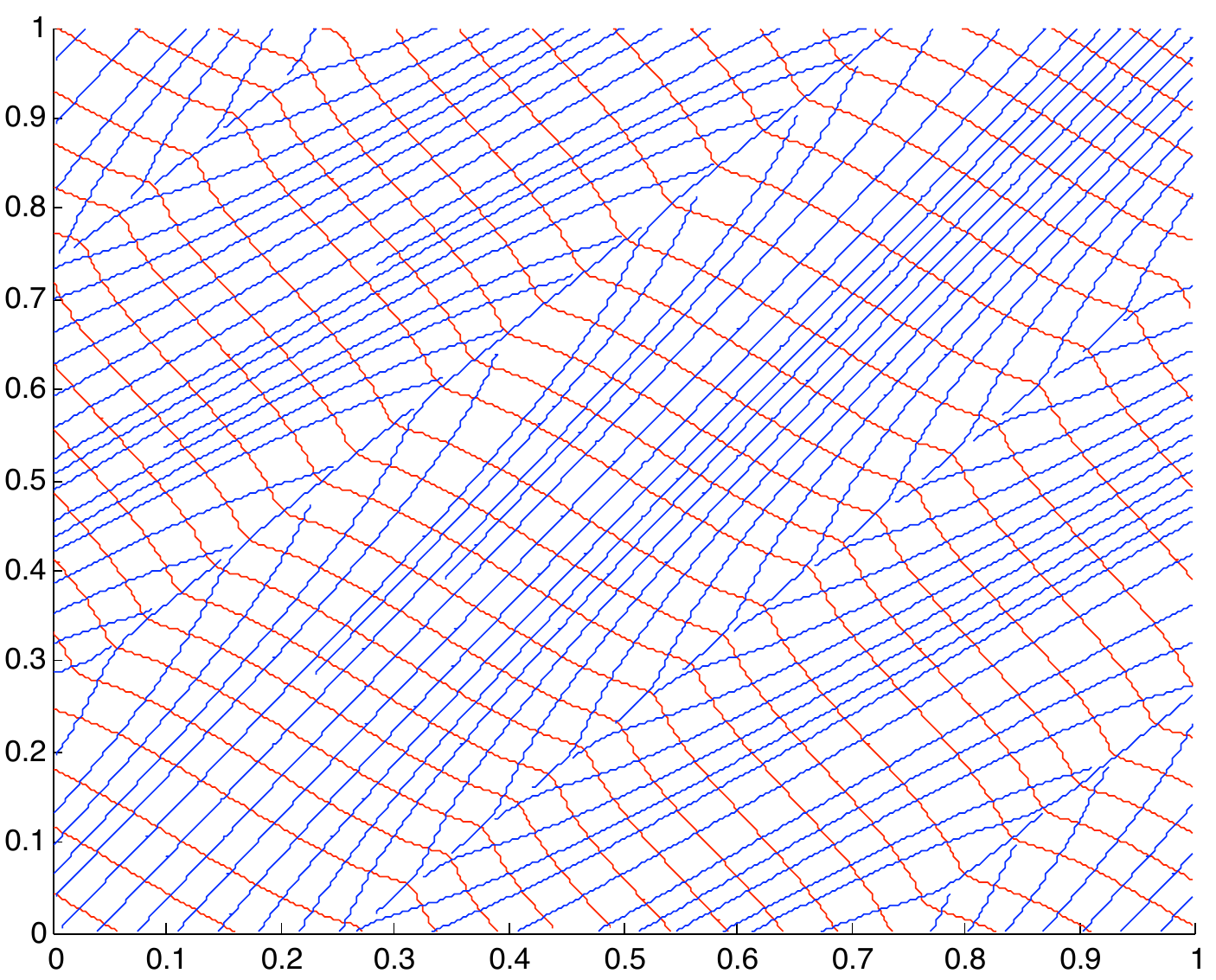}
     \end{center}
	 \caption{Schematic picture of the stable and unstable
	 foliations (\( k=1 \)) in forward time (left) and backward
	 time (right). }
	 \label{fig:foliations}
  \end{figure} 
   
We give here a preliminary discussion of the geometry of these
foliations. All statements are formally defined and proved in the
following sections.   
  
\subsubsection{The contracting foliation in forward time.} 
The contracting foliation \( \mathcal  E^{(1)}\) consists of
essentially horizontal leaves in most of the phase space, although 
close to the
horizontal lines  \( \{y=1/4\} \) and \( \{y=3/4\} \) 
the leaves exhibit a full
 horizontal ``fold''. The contracting direction \( e^{(1)} \) 
 is constant along horizontal lines and thus the tips of the 
 folds lie on a perfectly horizontal line. 
 Figure \ref{fig:foliations} is drawn for clarity
 for \( k=1 \) but in fact for larger \( k \) the folds occur in a
 very thin region. To get an idea, if we define a horizontal strip
 containing the folds in such a way that the boundaries of the strip
 are precisely the horizontal lines at which the contracting
 directions (tangent to the leaves of \( \mathcal E^{(1)} \)) are
 aligned with the positive and negative diagonals, 
 then the width of this strip is of order \( 1/k \).

\subsubsection{The expanding foliation in forward time}
 
The expanding foliation \( \mathcal F^{(1)} \) is everywhere orthogonal
to \( \mathcal E^{(1)} \) (since \( e^{(n)} \) and \( f^{(n)} \) are
always orthogonal; a general property of the
singular value decomposition).  It therefore has exactly two closed leaves given 
by the horizontal lines through the tips of the folds of the leaves of
\( \mathcal E^{(1)} \). All other leaves wrap around the torus 
accumulating on these two closed 
leaves as illustrated in Figure \ref{fig:foliations}. Notice that for 
large \( k \) these leaves are almost vertical moving very ``quickly''
from one folding region to the other, and then folding very sharply
and wrapping infinitely many times round the torus.

\subsubsection{The contracting foliation in backward time}
 This time the
contracting direction \( e^{(-1)} \) is constant along diagonals, i.e.
along lines of slope 1. The  foliation \( \mathcal E^{(-1)} \)
has two closed leaves: the two diagonals 
\( \{y-x=1/4\} \) and \( \{y-x=3/4\} \). The
other leaves of the foliations behave, from a topological point of
view, in a similar way to the leaves of the foliation \( \mathcal
F^{(1)} \), each one wrapping infinitely many times around the torus
and accumulating on closed
leaves. However, from a geometrical point of view, the situation is 
different. Rather than moving quickly from one closed leaf to 
the other, the leaves wrap around the torus many
times.  Moreover, the tangent directions, i.e. the directions of the
vectors \( e^{(-1)} \) change \emph{very rapidly} near the closed
leaves they , i.e. in a neighbourhood of the
closed leaf of width of the order of \( 1/k \): from an angle of about 
\( \pi/8 \) on one side of the closed leaf to an angle of about \(
3\pi/8 \) on the other side of the leaf. Then, more slowly, they
essentially re-align themselves with the diagonal in the region
between the closed leaves.

\subsubsection{The expanding foliation in backward time}
As before, the geometry of the expanding foliation \( \mathcal 
F^{(-1)} \) is everywhere orthogonal to \( \mathcal E^{(-1)} \) and
is completely determined by the geometry of \( \mathcal E^{(-1)} \).
The general direction of the leaves is along the negative diagonal,
exhibiting a peculiar \emph{bump} as they cross the 
two closed leaves of \( \mathcal E^{(-1)} \). This is related to the 
fact that, as observed in the previous paragraph, hyperbolic
coordinates vary very quickly in a neighbourhood of the 
closed leaves.

\subsection{The critical curve of tangencies}

We now describe the curve of tangencies between \( e^{(1)} \) and \(
e^{(-1)} \).  To obtain an explicit expression for this curve we need 
to define some constants. 
First of all we let 
\[
\delta^{*} :=\frac{1}{2\pi}\cos^{-1}\left(-\frac{1}{4\pi k}\right)
\quad \text{ and } \quad
\delta^{\pm} 
:= \frac{1}{2\pi}\cos^{-1}\left(-\frac{1\pm \sqrt{3}}{4\pi k}
\right)
    \]
 and 
 \[ 
 \hat\delta^{-}_{T} = \frac{1}{2\pi}\cos^{-1}
 \left(-\frac{1+\frac{\sqrt 3}{3}}{4\pi k}\right)
 \quad\text{and}\quad 
 \hat\delta^{+}_{T} = \frac{1}{2\pi}\cos^{-1}
 \left(-\frac{1+3\sqrt 3}{4\pi k}\right)
 \]
 Notice that  
 \[
 1+3\sqrt 3 > 1+ \sqrt 3 > 1+ \sqrt 3/3 
 > 1 >0 > 1- \sqrt 3 
 \]
 and therefore, for \( k \) large and since 
 the cosine function is decreasing near \( \pi/2 \) we have 
  \[
  \delta^{-}<1/4< \delta^{*}<  \hat\delta^{-}_{T} < \delta^{+} 
  <  \hat\delta^{+}_{T}
  \]
  and all five constants tend to \( 1/4 \) as \( k\to\infty \).
 Now  let 
  \[ \hat\Delta_{T}= [\hat\delta^{-}_{T}, \hat\delta^{+}_{T}]
      \cup [1-\hat\delta^{+}_{T}, 1-\hat\delta^{-}_{T}]  \]
denote the union of two horizontal strips defined by the above
constants.  It will be useful to introduce an alternative coordinate
system by writing 
\[ 
\tilde y = y-x \mod 1.
\] 
Geometrically, lines of the form \( \{\tilde y = c\} \) 
are lines of constant slope equal to \( +1 \) which intersect 
the \( y \)-axis at \( y=c \). 
Using this notation we can define 
\[ 
\tilde\psi_{c} = \tilde\psi_{c}(\tilde y) = 2\pi k\cos (2\pi \tilde y)  
\quad\text{
and}\quad
 \tilde\varphi = \tilde\varphi (\tilde\psi_{c}) = 
  -\frac{2(\tilde\psi_{c}^{2}+\tilde\psi_{c}+1)}{1+2\tilde\psi_{c}}
 \]
and 
\[   \tilde\Phi (\tilde\varphi ) = 
\frac{1}{2\pi}\cos^{-1}\frac{-(\tilde\varphi+2)
\pm\sqrt{3\tilde\varphi^2 + 4}}{4\pi k \tilde\varphi}
 \]
We note that the expression under the square root is always positive
and therefore, as we shall see below,  
the function \( \Phi(\tilde\psi_{c}) \) is multivalued associating exactly 4
values to each \( \tilde y \).  Both \( \tilde\varphi \) and \(
\tilde\Phi \) are ultimately functions of \( \tilde y \) and therefore
we shall sometimes write them as \( \tilde\varphi(\tilde y) \) and \( 
\tilde\Phi(\tilde y) \). With this notation we define the possibly
multivalued function 
\[ 
\Gamma(\tilde y) = \begin{cases}
\Phi (\tilde y) \cap 
[\delta^{-}, \delta^{+}] \cup 
[1-\delta^{+}, 1-\delta^{-}] & \text{ if } \tilde y\in [0,
\delta^{*}]\cup [1-\delta^{*}, 1]\\
\Phi (\tilde y) \cap 
[\delta^{+}, 1-\delta^{+}] & \text{ if } 
\tilde y \in 
[\delta^{*}, 1-\delta^{*}].
\end{cases}
\]
 We shall show below that \( \Gamma \) takes on exactly two values
 for each value of \( \tilde y \). 
\begin{theorem}\label{th:tang}
    The set \( \mathcal C^{(1)} = \{\text{points of tangencies of } \mathcal E^{(1)}
\text { and }  \mathcal E^{(-1)} \} \) consists of two smooth curves
contained in  \( \hat\Delta_{T} \) and given by the graph of the
function 
\(  y=\Gamma(\tilde 
y). \)
\end{theorem}   
 This gives an explicit, albeit not particularly user-friendly, 
 formula for \( \mathcal C^{(1)} \). In the course 
 of the argument used to derive it, we shall obtain more information
 about its properties and will be able to deduce several facts about
 the geometry of the curve. In particular, we shall define below some 
 additional constants 
 \(
 \hat\delta_{T}^{-}, \delta_{T}^{-}, \delta_{T}^{+}, \hat \delta_{T}^{+} 
 \)
 satisfying 
 \[
 \delta^{-}< 1/4 < \delta^{*} < \hat\delta_{T}^{-} < \delta_{T}^{-} < 
     \delta^{+} < \delta_{T}^{+}< \hat \delta_{T}^{+} 
  \]
with all constants tending to \( 1/4 \) as \( k\to \infty \), as use
these constants to sketch a detailed picture of the curves of
tangencies.
\begin{figure}[h]
 \includegraphics[width=10cm]{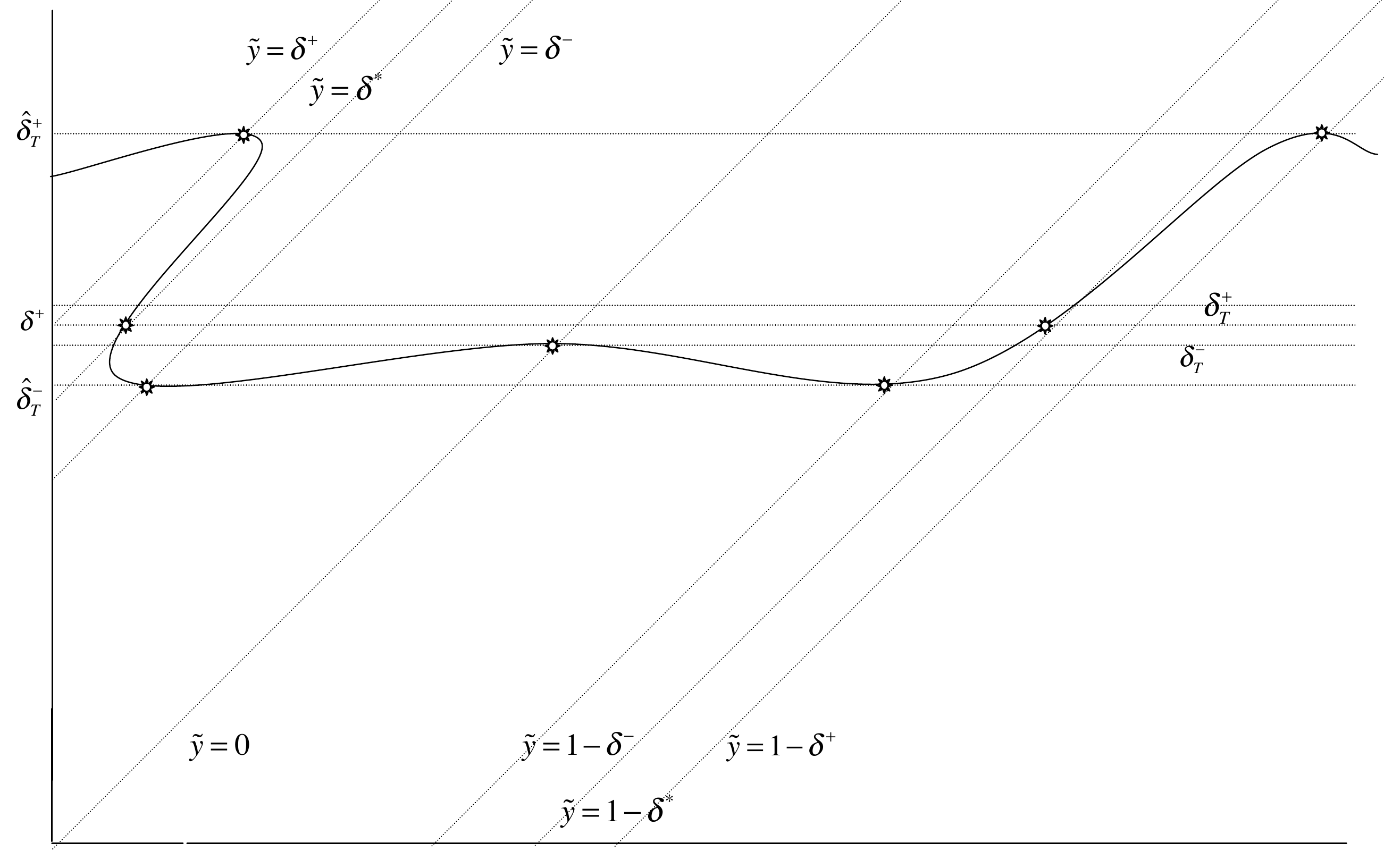}    
 \caption{Part of the curve of tangencies}
 \label{tangencygraph}
 \end{figure}
In particular we shall show that \( \mathcal C^{(1)} \) 
contains the following points of tangencies \( P_{i}(\tilde y, y) \):
\begin{center}
\begin{tabular}{c | c ccccccc}
    & \( P_{1} \)  & \( P_{2} \)  & \( P_{3} \)  & \( P_{4} \) 
    & \( P_{5} \)  & \( P_{6} \)  & \( P_{7} \)  & \( P_{8} \) \\
    \hline
    \(  \tilde y \) &  \( 0 \) & \( \delta^{-} \) & \( \delta^{*} \) & \(
     \delta^{+} \) & \( 1/2 \) & \( 1-\delta^{+} \) 
     & \( 1-\delta^{*} \) & \( 
     1-\delta^{-} \)  \\
    \( y \) & \( \delta^{-}_{T} \) & \( \hat\delta^{-}_{T} \) & \( \delta^{+}  \)
    & \(  \hat\delta^{+}_{T}  \) & \( \delta^{+}_{T}  \) & 
    \(  \hat\delta^{+}_{T}  \) &  \( \delta^{+} \) & 
    \(  \hat\delta^{-}_{T} \) \\
    \hline
 \end{tabular}
\end{center}
The relative positions of these points is illustrated in Figure
\ref{tangencygraph}, notice that the points are indexed according to
their ordering along the curve. The relative positions of the
horizontal coordinates of the points as as illustrated, in particular \( 
P_{2} \) is the point with the smallest horizontal coordinate,
followed by \( P_{3} \) followed by \( P_{1} \).

\subsection{Uniform hyperbolicity away from the critical curves}
We  recall the basic definitions of a hyperbolic structure. 
A \emph{cone} \( C\subset \mathbb R^{2} \) is a closed 
convex union of 
one-dimensional linear subspaces of \( \mathbb R^{2} \). 
For a subset \( \Lambda \subseteq \mathbb T^{2} \), a 
\emph{conefield} over \( \Lambda \) is a family 
\[ 
\mathfrak{C} = \{C(z) \subset T_{z}\mathbb T^{2}\}_{z\in\Lambda}
\]
of cones lying in the tangent spaces of points in \( \Lambda \). We
say that the conefield is \emph{continuous} if the boundaries of the
cones vary continuously with the point \( z \). We say that the
conefield \( \mathfrak C \) is (strictly) 
\emph{forward invariant} under the 
map \( f \) if  
\[ 
Df_{z}(C(z))\subset C(f(z))
\] 
for any \( z \in \Lambda \) for which \( f(z) \in \Lambda \). Notice
that we do not necessarily assume that \( \Lambda \) is forward
invariant.  We say that the conefield \( \frak C \) is
\emph{backward invariant} if the \emph{complementary conefield} is 
``forward'' invariant by \( Df^{-1} \), i.e. if 
\[ 
Df^{-1}_{z}(T_{z}\mathbb T^{2}\setminus C(z)) \subset 
T_{f^{-1}(z)}\mathbb T^{2} \setminus C(f^{-1}(z))
\]
for all \( z\in \Lambda \) for which \( f^{-1}(z)\in \Lambda \). 
We say that the (not necessarily invariant) subset \( \Lambda\subset
\mathbb T^{2} \) has a \emph{hyperbolic structure} if there exists and
conefield \( \mathfrak C \) over \( \Lambda \) which is forward and
backward invariant and which satisfies the following 
uniform hyperbolicity properties: there exist constants \( C, \lambda 
> 0 \) such that 
\[ 
\|Df^{\pm n}_{z}(v^{\pm})\|\geq C e^{\lambda n}\|v^{\pm}\|
\]
for vectors \( v^{+}\in C(z) \), \( v^{-}\in T_{z}\mathbb
T^{2}\setminus C(z) \) and \( n\geq 1 \) such that 
\( z, f(z),\ldots, f^{\pm(n-1)}(z)\in\Lambda \).

We now define the regions outside which we have uniform hyperbolicity.
For  \( 1 \leq m <  k \) let
\[ \delta^{(\pm m)} = \frac{1}{2\pi}\cos^{-1}\frac{m}{\pi k} \]
and 
\[ 
\Delta^{(m)} = \{(x,y): y \notin [\delta^{(-m)}, \delta^{(m)}] 
\cup [1-\delta^{(-m)}, 1-\delta^{(m)}].
\]
We note that for \( m \) relatively small, \( \Delta^{(m)} \) is not
much larger than the strips defined above containing the critical
curve. For \( m \) relatively large, say \( m \approx k^{1-\epsilon} \)
the region  \( \Delta^{(m)} \) is relatively 
much larger than those strips but still very small. 

For an arbitrary point \( z\in\mathbb T^{2} \) and unit vector 
\( v= (\cos \theta, \sin\theta)  \)  we write
. 
The following result says that there is a uniformly hyperbolic 
structure outside \( \Delta^{(m)} \) with minimum expansion rate \( m \).
We suppose here that some large parameter value \( k \) has been fixed.

\begin{theorem}\label{th:hyp} 
Suppose that  \( k > m \geq  2 \), 
\( z\in \mathbb T^{2}\setminus \Delta^{(m)} \) and 
\( v= (\cos \theta, \sin\theta)  \) 
with \( \tan\theta  \in (m^{-1},
m)\).  
Then, letting 
\( Df_{z}(v)  = \tilde v = (\cos\tilde\theta, \sin\tilde\theta)  \)
we have 
 \[ \tan\tilde\theta \in (1-m^{-1}, 1+m^{-1}) \quad\text{ and } \quad 
 \|\tilde v \|\geq m.
 \]
\end{theorem}

This immediately gives a hyperbolic structure outside \( \Delta^{(m)} \)
with increasingly strong hyperbolicity as \( m \) is chosen large.

\section{Forward time foliations}
\label{forfol}
\subsection{Definitions and notation}
We shall suppose throughout the paper that \( k \) is fixed and
therefore generally omit it as a subscript in the notation.   
We shall repeatedly be calculating inverses of trigonometric
functions and we therefore fix once and for all the
ranges of the inverse functions  as follows: 
\[ 
\sin^{-1}: [-1,1] \to [-\pi/2, \pi/2], 
\quad \cos^{-1}: [-1, 1] \to [0, \pi], 
\quad 
\tan^{-1}: \mathbb R \to [-\pi/2, \pi/2].
\]
We shall often write unit  vectors in the form 
\[
v=(\cos \theta, \sin \theta)
\] 
where we always consider \( \theta \) to be in the
range \( [-\pi, \pi] \). We define two polynomial functions 
\begin{equation}\label{poly}
\mathcal P_{1}(x) = 2x^{2}+2x -1 
\quad \text{ and } 
\quad 
\mathcal P_{2}(x) = x^{2}+x+1.
\end{equation}
Notice that their derivatives are 
\[ 
\mathcal P_{1}'(x) = 4x+2 
\quad\text{ and } \quad \mathcal
P_{2}'(x) = 2x+1.
\]
In particular, \( \mathcal P_{1}'= 2 \mathcal P_{2}' \). 
Moreover, \( \mathcal P_{1} \) has two real solutions given by 
\[
x= \frac{-1\pm\sqrt 3}{2}
\]
whereas \( \mathcal P_{2} \) has no real solutions.
We also define 
\[ 
\psi_{c}=\psi_{c}(y)= 2\pi k\cos (2\pi y) 
\quad \text{ and } \quad 
\psi_{s}= \psi_{s}(y)=2\pi k\sin (2\pi y)
\]
We then write the derivative of \( f \) as 
\begin{equation}\label{Df}
Df{(x,y)} = \begin{pmatrix}1&\p\cos(2\pi y)\\1&1+\p\cos(2\pi
y)\end{pmatrix}
= \begin{pmatrix}1&\psi_{c}  \\1&1+ 
\psi_{c} \end{pmatrix}
\end{equation}
We also let 
\[ 
\varphi = \varphi(y) = 
-\frac{4\psi_{c} + 2}{2\psi_{c}^{2}+2\psi_{c}-1} = -\frac{\mathcal
P_{1}'(\psi_{c})}{\mathcal P_{1}(\psi_{c})}.
\]
 \begin{figure}[h]
     \begin{center}
     \includegraphics[height=4cm, width=16 cm]{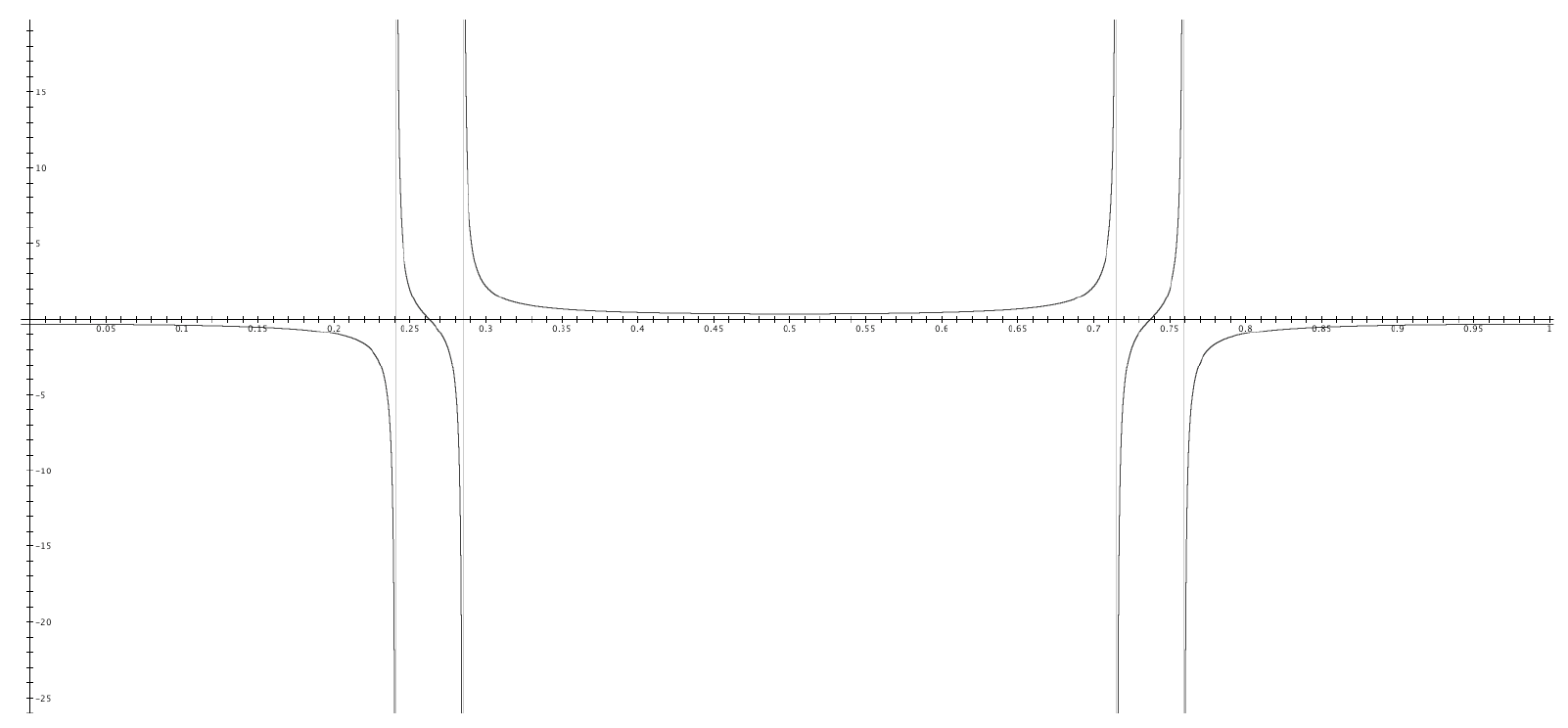}
     \end{center}
     \caption{Graph of \( \varphi(y) \). }\label{varphi}
     \label{graphvarphi}
  \end{figure} 
The function \( \varphi \) arises naturally and plays an important
role in the 
computation of hyperbolic coordinates 
and therefore we mention some key properties in the following 
    \begin{lemma}
    \( \varphi \) has 
   zeros at \( \delta^{*} \) and \( 1-\delta^{*} \), 
  asymptotes at \( \delta^{-}, \delta^{+}, 1-\delta^{+},
    1-\delta^{-} \), where  
\[
     \delta^{\mp} := \frac{1}{2\pi}\cos^{-1}\frac{-1\pm \sqrt{3}}{4\pi k}
    \quad 
    \text{ and } 
    \quad 
     \delta^{*} :=\frac{1}{2\pi}\cos^{-1}\frac{-1}{4\pi k}.
    \]
\( \varphi \) also 
    turning points at \( 0 \) and \( 1/2 \), with 
 \( 
    \varphi(0) \approx -1/k \) and \(
    \varphi(1/2) \approx 1/k \).
    \end{lemma}

    \begin{proof}
  Zeroes are given by solutions of \( \mathcal P_{1}'(\psi_{c})=0 \). 
  Thus gives 
   \( 2\psi_{c} + 1  = 0 \) and \( 2\pi k\cos (2\pi y) =
     -1/2 \), and therefore    \(
     y = \frac{1}{2\pi}\cos^{-1} \frac{-1}{4\pi k}
      \).     
  Asymptotes are given by  solutions of 
  \(  \mathcal P_{1}(\psi_{c})= 2\psi_{c}^{2}+2\psi_{c} -1 = 0
   \)
  and therefore are solutions 
    of  
 \(  
     \psi(y) = (-1\pm
     \sqrt 3)/2 
\) 
or \( 
     y= \frac{1}{2\pi}\cos^{-1} \frac{-1\pm \sqrt 3}{4\pi k},
     \) 
     which are exactly \(
     \delta^{-} \) and \( \delta^{+} \). 
  Finally, differentiating \( \varphi(y) \) we get 
  \begin{equation}\label{phiprime}
   \varphi' = \frac{d\varphi}{dy}=
   \frac{d\varphi}{d\psi_{c}}\frac{d\psi_{c}}{dy}
 = -\frac{\mathcal P_{1}''\mathcal P_{1}-\mathcal P_{1}'^{2}}
 {\mathcal P_{1}^{2}}\psi'_{c}
  = \frac{8 \mathcal P_{2}(\psi_{c})}
  {\mathcal P_{1}(\psi_{c})^{2}}\psi_{c}'
   \end{equation}
where  the last equality comes from the fact that 
\( \mathcal P_{1}''\mathcal P_{1}-\mathcal P_{1}'^{2} = 
-4(2\psi_{c}^{2}+2\psi_{c}-1)+(4\psi_{c}+2)(4\psi_{c}+2)
= 8(\psi_{c}^2+\psi_{c}+1) \). 
Since \( \mathcal P_{2} \) has no real solutions, 
turning points of \( \varphi \) are just the solutions to 
\( \psi_{c}'(y) = 0 \) which is equivalent to \( \sin 2\pi y = 0 \)
whose solutions are \( y=0 \) and \( y=1/2 \). We have 
\( \psi_{c}(0) = 2\pi k \) and \( \psi_{c}(1/2) = -2\pi k \) and so 
	\begin{equation}\label{minmaxvarphi}
\varphi(0) =  -\frac{8\pi k+2}{8\pi^{2}k^{2}+2\pi k -1}
\approx -\frac{1}{k}
\quad \text{ and } 
\varphi (1/2) = -\frac{-8\pi k+2}{8\pi^{2}k^{2}-2\pi k -1}
\approx \frac{1}{k}.
	\end{equation}
\end{proof}
We also let 
\[ 
\Delta := [\delta^{-}, \delta^{+}].  
\quad\text{ and } \quad 
1-\Delta := [1-\delta^{+}, 1-\delta^{-}].
\]
We note that 
\(
0<\delta^{-} < 1/4 < \delta^{*}< \delta^{+}<1/2 
\)
and 
\( 
\delta^{-},  \delta^{*}, \delta^{+}\to 1/4 
\text{ as }  k \to\infty\). In particular \( \Delta \) and \( 1-\Delta \)
shrink as \( 1/k \) increases and converge to the 
lines \(\{y= 1/4\} \) and \( \{y=3/4\} \) respectively.

\subsection{First order hyperbolic coordinates in forward time}

We are now ready to compute precisely the direction of hyperbolic
coordinates in forward time.  First of
all we define 
\begin{equation}
 \theta^{(1)}(y) = \begin{cases}
 \pi+ \frac{1}{2}\tan^{-1} \varphi(y) &\text{ if } y\in [0, \delta^{-}]
 \cup [1-\delta^{-}, 1]\\
 \frac{\pi}{2}+ \frac{1}{2}\tan^{-1} \varphi(y) &\text{ if } 
 y\in [\delta^{-}, \delta^{+}]\cup 
 [1-\delta^{+}, 1-\delta^{-}]\\
\frac{1}{2}\tan^{-1} \varphi(y) &\text{ if } y\in \left[\delta^{+},
1-\delta^{+}\right].
\end{cases}
\end{equation}
Then we have the following
\begin{proposition}\label{e1+}
For every \( k> 0 \), 
hyperbolic coordinates for \( f_{k} \) are defined at every point
  of \( (x,y) \in \mathbb T^{2} \) and depend only on \( y \). The
  vector 
 \[
e^{(1)}(y)=(\cos \theta^{(1)}(y), \sin\theta^{(1)}(y) )
\]  
is a unit vector in the most contracted direction for \( Df_{k} \) 
  at \( (x,y) \).  The most expanded direction is everywhere orthogonal to
  \( e^{(1)} \). 
\end{proposition}
A graph of \( \theta^{(1)}(y) \) is shown below. Once again this is
shown for clarity for small \( k \), as \( k \) increases the graph
approaches a step function.
\begin{figure}[h]
 \begin{center}
     \includegraphics[width=16cm, height=4cm]{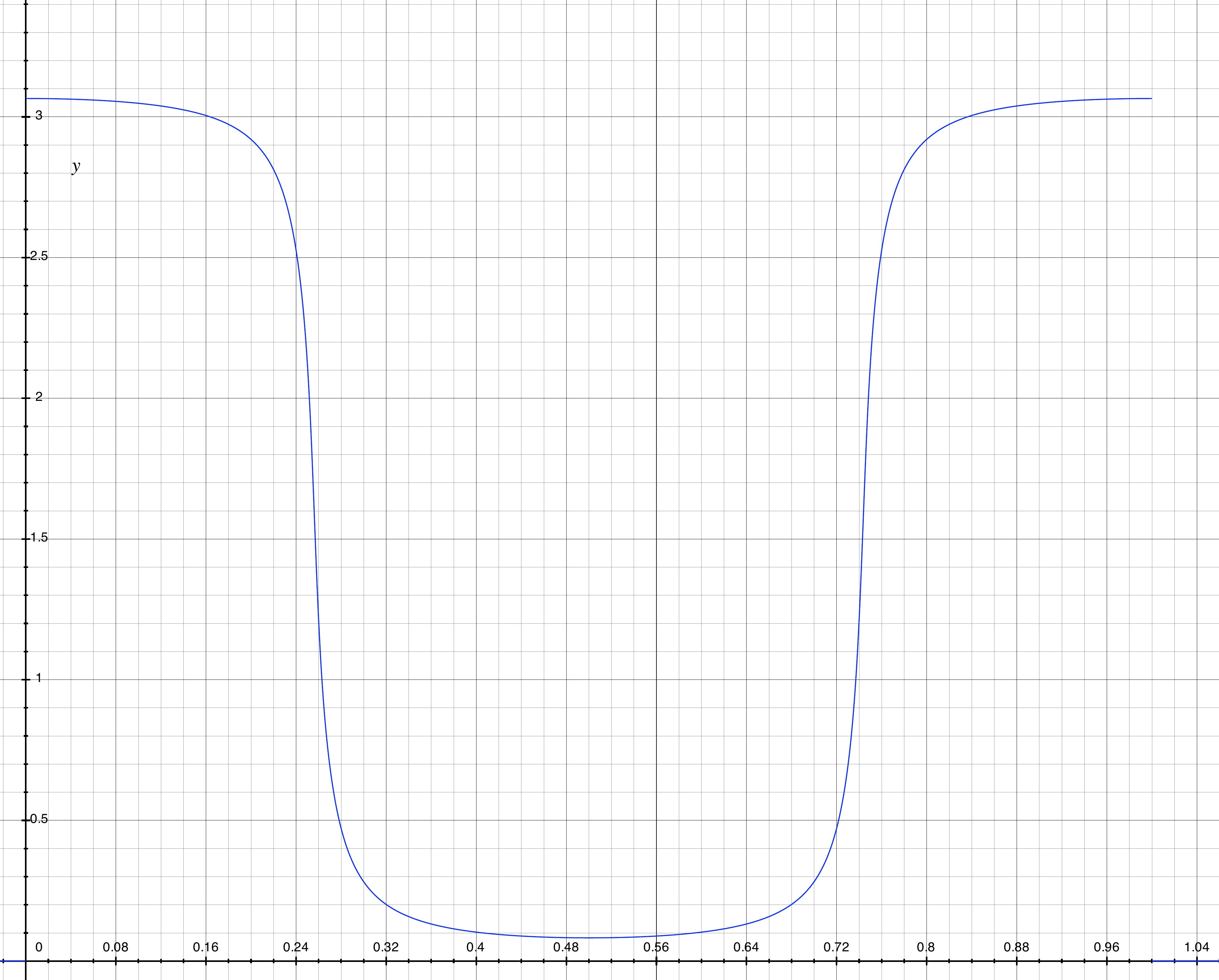}
 \end{center}
\caption{Graph of \( \theta^{(1)}(y) \)} 
\label{theta1}
\end{figure}

An analysis of the properties of the function \( \theta^{(1)} \)
  gives the following 
  
  \begin{corollary}\label{e1+cor}
 \( e^{(1)} \) rotates monotonically
  \textbf{clockwise} in \( [0,1/2] \) between an angle of \( \theta^{(1)}(0)
  \approx \pi - 1/k \) close to the negative horizontal semi-axis at \( 
  y=0 \) to an angle of \( \theta^{(1)}(1/2) \approx 1/k \) close to
  the positive horizontal semi-axis at \( y=1/2 \), 
\textbf{swinging rapidly} between the negative diagonal at
  \( y=\delta^{-} \), through the vertical at \( y=\delta^{*} \), to the
  positive diagonal at \( y=\delta^{+} \). 
  Then 
  \( e^{(1)} \) rotates  monotonically
   \textbf{counter-clockwise} in \( [1/2, 1] \) between an angle of 
   \( \theta^{(1)}(1/2) \approx 1/k \) close to
   the positive horizontal semi-axis at \( y=1/2 \), 
   to and angle  \( \theta^{(1)}(1) 
   = \theta^{(1)}(0) \approx \pi - 1/k \) close to the negative horizontal 
   semi-axis at \( 
   y=1=0 \), \textbf{swinging rapidly} between the 
   positive diagonal at \( y=1-\delta^{+} \), 
   through the vertical at \( y=1- \delta^{*} \), to the
   negative diagonal at
   \( y=1- \delta^{-} \).
\end{corollary}

\begin{proof}[Proof of Proposition \ref{e1+}]
    We start by showing that hyperbolic coordinates exist for all \( k \)
    and for all points in \( \mathbb T^{2} \). 
We then give a general formula for hyperbolic coordinates (which 
implies in particular that the most contracted and the most expanded
directions are always orthogonal). We then substitute the explicit
expressions related to our setting into this formula. Finally we
establish which of the directions given by the formula is actually
the contracting direction and which is the expanding direction. 

\subsubsection{Existence of hyperbolic coordinates}

Since $\det(Df(z))\equiv 1$ it is sufficient
to show that for every \( k \) and for every point \( z\in \mathbb T^{2} \)
there is \emph{some}
vector $v\in T_z\mathbb T^{2}$ which is expanding. Indeed, for any \( 
k \) and any \( z \) we have  
\[ 
\|Df_{z}(1,0)\| = \left\|\begin{pmatrix} 1 & \psi_{c} \\ 1 & 1+
\psi_{c}\end{pmatrix} \begin{pmatrix} 1 \\ 0 \end{pmatrix}\right\| =
\|(1,1)\| = \sqrt{2} > 1.
\]

\subsubsection{General formula for hyperbolic coordinates}
Given \( z\in \mathbb T^{2} \), 
contracting and expanding directions are computed
exactly as follows. 
Letting 
\(
v=(\cos \theta, \sin \theta)
\) denote a general vector, 
the most contracted and most expanded directions  are 
solutions to
the differential equation 
\[
\frac{d\|Df_{z}(\sin \theta, \cos \theta)\|}{d\theta} = 0. 
\] 
If the left hand side of this equation is not identically equal to zero, we 
obtain the relation
\begin{equation}\label{contrdir}
\tan 2\theta  =
\frac{2 (\pfi x1\pfi y1 +\pfi x2\pfi y2)}
{(\pfi x1)^2+(\pfi x2)^2 - (\pfi y1)^2 -(\pfi y2)^2}.
\end{equation}
Notice that \( \tan2\theta \) is a periodic function with period \(
\pi/2 \) and therefore defines 2 orthogonal directions; it does not,
however,  distinguish the most contracting from the
 most expanding direction. 

\subsubsection{Specific formula for hyperbolic coordinates}
We can now substitute the partial derivatives of \( f \)
 into \eqref{contrdir}. Since the derivative of \( f \) depends only
 on the \( y \) coordinate, it follows that the hyperbolic coordinates
 depend only on \( y \). 
 Using the definition of \( \varphi \), we get 
\begin{equation}\label{e1forward}
\tan{2\theta(z)} =  
\frac{2(\psi_{c}+1+\psi_{c})}{1+1-\psi_{c}^{2}-(1+\psi_{c})^{2}}
=\frac{2(2\psi_{c} +1)}{1-2\psi_{c} (\psi_{c}
+1)}=\varphi(y) 
\end{equation}
 As mentioned above, this equation defines two
orthogonal directions without distinguishing the expanding and the
contracting one. Naively inverting \( \tan \) to get 
\begin{equation}\label{tan}
\theta = \frac{1}{2}\tan^{-1}\varphi(y)  
\end{equation}
fixes one of these solutions, namely the one belonging to \( (-\pi/4,
\pi/4) \), which, in general, may be a contracting or an expanding
direction.  
Moreover, as \( y \) changes, the hyperbolic coordinates change and
equation \eqref{tan} may pick out the contracting directions for some 
values of \( y \) and the expanding directions for others. 

\subsubsection{Establishing the contracting and the expanding
directions}
We fix first of all a single value of \( y \), for simplicity let \(
y=0 \). From \eqref{minmaxvarphi} we have \( \varphi (0) \approx -1/k \)
for large \( k \), and therefore, from \eqref{tan}, 
\[ 
\theta(0) = \frac{1}{2}\tan^{-1}\left(-\frac{1}{k}\right) 
\approx -\frac{1}{k}.
\]
This means that one of the solutions of \eqref{e1forward} has very
small negative slope. A simple calculation shows that this is indeed
the contracting direction. To see this we shall show 
that  every vector \(
v=(\cos\theta, \sin\theta) \) with \( \theta\in (\pi/4, 3\pi/4) \) is 
significantly expanded: recall first that for \( y=0 \) we
have \( \psi_{c}(0) = 2\pi k \cos 0 = 2\pi k \), then 
for all \( \iota\in (-1, 1) \) we have 
\[ 
\|Df_{y=0}(\iota, 1)\| = 
\left\|\begin{pmatrix} 1 & \psi_{c} \\ 1 & 1+
\psi_{c}\end{pmatrix} \begin{pmatrix} \iota \\ 1 \end{pmatrix}\right\| 
=
\left\| 
\begin{pmatrix}
\iota + 2\pi k \\ 
\iota+1+2\pi k
\end{pmatrix} \right\| 
\gtrsim k
\]
This implies that the contracting direction is actually almost
horizontal (and with negative slope). We now have two choices for
defining \( \theta^{(1)} \) and the
unit vector \( e^{(1)} \). It seems simpler, for the rest of the
discussion to choose 
\( \theta^{(1)} = \pi + \frac{1}{2}\tan^{-1}\varphi(0) \) giving a
most contracted unit vector close to the negative horizontal semi-axis.

\subsubsection{Analytic continuation of the most contracted direction}
To obtain a complete formula for \( \theta^{(1)} \) we need to keep
track of the rotation of \( e^{(1)} \) as \( y \) increases. Indeed,
if the hyperbolic coordinates changed direction so that the most
expanded direction corresponded to angles \( \theta\in (-\pi/4, \pi/4) \)
then the formula would start picking up the wrong direction. 

We observe therefore first of all that \( \varphi \) is decreasing in \( 
y \) at \(  y=0\) and that this corresponds to a clockwise rotation of
\( e^{(1)} \) as \( y \) increases near 0. For \( y=\delta^{-} \) we
have \( |\varphi(y)| = \infty \) which corresponds to the hyperbolic
coordinates being aligned with the diagonal. As \( y \) crosses \(
\delta^{-1} \),  the
equation \( \frac{1}{2}\tan^{-1}\varphi(y) \) which always only picks 
up the solution inside \( [-\pi/4, \pi/4] \) switches from
picking up one of the 
most contracting directions to picking up one of 
the most expanding directions. 
To realize the continuity of \( e^{(1)} \) 
(hyperbolic coordinates depend smoothly on the position) we 
 change the definition to \( \theta^{(1)}=
\frac{\pi}{2}+\frac{1}{2}\tan^{-1}(\varphi(y)) \). A similar change
occurs as \( y \) crosses \( \delta^{+} \) and then \( 1-\delta^{+} \)
and \( 1-\delta^{-} \).
\end{proof}

\subsubsection{Geometry of the most contacting direction}

A more detailed analysis along the lines of the arguments given above,
allow us to give a detailed description of how that most contracting
unit vector \( e^{(1)} \) depends on the point \( y \).

\begin{proof}[Proof of Corollary \ref{e1+cor}]
 To find the turning points of \( \theta^{(1)} \) we    
calculate derivatives with respect to \( y \) and get 
\[ 
\frac{d\theta^{(1)}}{dy}=\frac{1}{2}\frac{\varphi'}{(1+\varphi^2)}.
\]
From the proof of Proposition \ref{e1+} we have that 
 the only zeros of \( \varphi' \) are 
 \( y=0 \) and \( y=1/2 \).  Moreover, 
 at  \( y=0 \) we have \( \psi_{c}(y) = 2\pi k \cos (2\pi y) = 2 \pi k \)
       and therefore 
       \[ 
	\theta^{(1)}(0)
	=\pi+\frac{1}{2}\tan^{-1}\frac{8\pi k + 2}{1-8\pi^{2}k^{2}-4\pi k}
	\approx \pi - \frac{1}{k}
       \]
     and st \( y=1/2 \) we have \( \psi_{c}(y)  = -2 \pi k \) and therefore
     \[ 
      \theta^{(1)}\left(\frac{1}{2}\right) 
      = \frac{1}{2}\tan^{-1}\frac{-8\pi k + 2}{1-8\pi^{2}k^{2}+4\pi k} 
      \approx \frac{1}{k}
       \]  
    For \( y \) equal to \( \delta^{\pm}, 1-\delta^{\pm} \) we have \(
    |\varphi(y)| = \infty \) which corresponds precisely to being aligned 
    with the diagonals as stated. For \( y \) equals \( \delta^{*} \) and
    \( 1-\delta^{*} \) we have \( \varphi(y)=0 \) which corresponds
    precisely to \( e^{(1)}(y) \) being vertical as stated.
    
 \end{proof}

\subsection{First order foliations in forward time}

The geometry of the stable and unstable foliations \( \mathcal E^{(1)} \)
and \( \mathcal F^{(1)} \) now follows by a careful consideration of
the position of the vectors \( e^{(1)} \) as described in Corollary
\ref{e1+cor}.

\section{Backward time foliations}
\label{bacfol}
We now carry out an analogous analysis of the inverse map
\begin{equation}\label{finv}
f_{k}^{-1}\begin{pmatrix}x\\ y\end{pmatrix}=
\begin{pmatrix}x-k\sin 2\pi(y-x)  \\ y-x 
\end{pmatrix} \mod 1
\end{equation}
in order to
obtain the contracting and expanding foliations \( \mathcal E^{(-1)} \)
and \( \mathcal F^{(-1)} \) in backward time.

\subsection{Definitions and notation}
Hyperbolic coordinates in backward time are  constant along
``diagonals'', i.e. lines of slope 1. We therefore start by
introducing a notations which will simplify our
analysis: let 
\[ 
\tilde y = y-x \mod 1.
\] 
Lines of the form \( \{\tilde y = c\} \) have unit slope and  intersect 
the \( y \)-axis at \( y=c \). 
Thus, when it is convenient to do so, 
we will write the coordinates of a point 
in the form  \( (\tilde y, y) \)
meaning that this point lies at the intersection of the horizontal
line through \( (0,y) \) (in standard coordinates) 
with the diagonal through \( (0, \tilde y) \) (in standard
coordinates). We shall often switch between standard and diagonal
coordinates sometimes even using both within the same expression. To
avoid confusion we shall use a \ \( \tilde{} \) \  to indicate
variables and
functions of variables in ``diagonal coordinates''. We let 
\[ 
\tilde\psi_{c}= 2\pi k\cos (2\pi \tilde y) 
\quad \text{ and } \quad 
\tilde \psi_{s}= 2\pi k\sin (2\pi \tilde y)
\]
and write the derivative of \( f^{-1} \) as 
\[ 
Df_{k}^{-1}(x,y) =
\begin{pmatrix}1+\p\cos(2\pi(y-x))&-\p\cos(2\pi(y-x))\\-1&1\end{pmatrix}
    =
\begin{pmatrix}1+ \tilde \psi_{c} & - \tilde 
\psi_{c} \\ -1 & 1 
\end{pmatrix}.
\]
We also define 
\[   \tilde\varphi = \tilde\varphi (\tilde y) = 
  -\frac{2(\tilde\psi_{c}^{2}+\tilde\psi_{c}+1)}{1+2\tilde\psi_{c}} 
  = -\frac{2\mathcal P_{2}(\tilde\psi_{c})}{\mathcal
  P_{2}'(\tilde\psi_{c})}\]  
  The graph of \( \tilde\varphi(\tilde y)) \) is sketched in Figure
  \ref{tildevarphik1} below.
  \begin{figure}[h]
      \begin{center}
      \includegraphics[width=16cm, height=4cm]{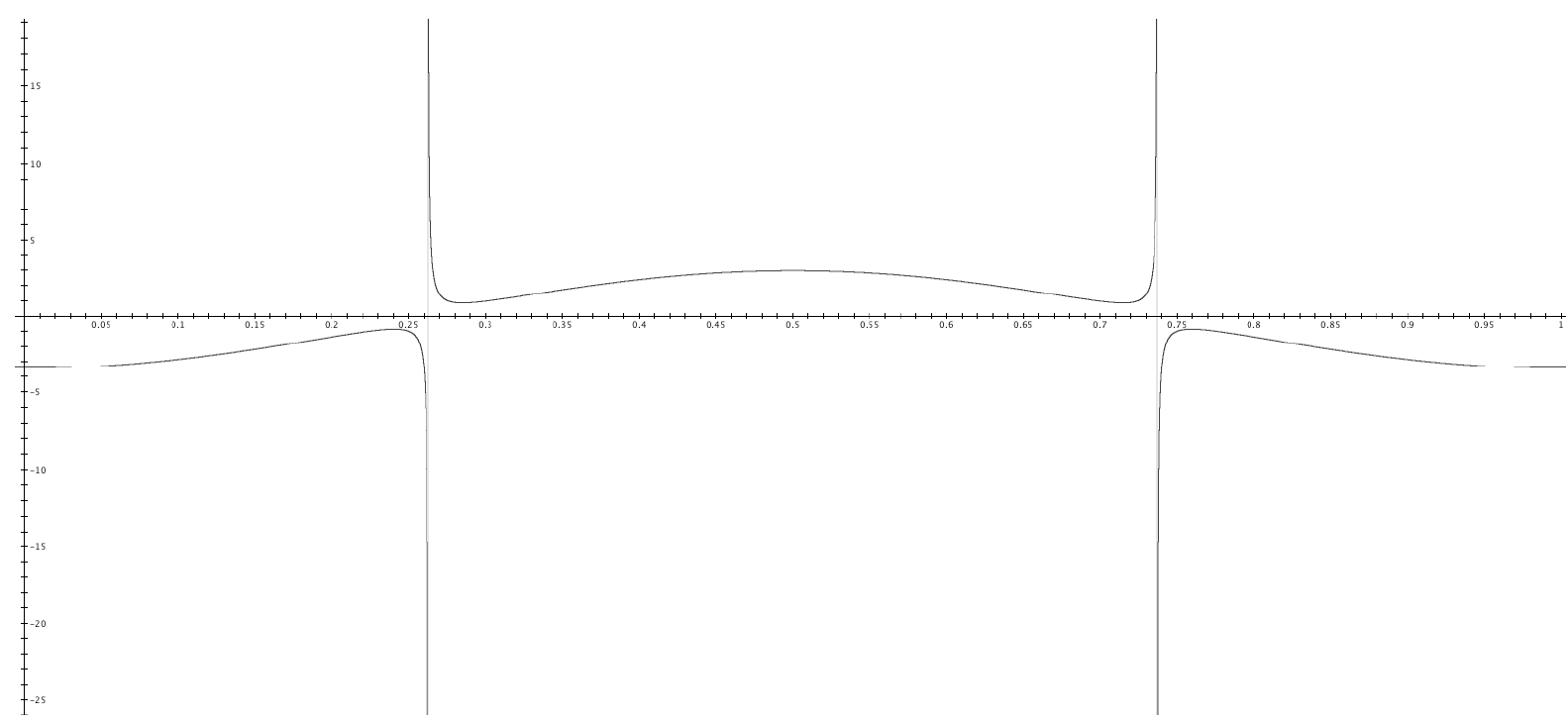}
      \caption{Graph of \( \tilde\varphi (\tilde y) \).}
      \label{tildevarphik1}
   \end{center}   
    \end{figure}  
    \begin{lemma}\label{l:tildevarphi}
     \( \tilde\varphi \) has no zeroes, two asymptotes, at \( \delta^{*} \)
     and \( 1-\delta^{*} \), and six turning points, at \( 0, \delta^{-},
     \delta^{+}, 1/2, 1-\delta^{+}, 1-\delta^{-} \), with 
     \[ 
     \tilde\varphi (0) \approx -k, 
     \quad
     \tilde\varphi (1/2) \approx k, 
 \] and \[
     \tilde\varphi(\delta^{-}) = \tilde\varphi(1-\delta^{-}) =
     -\frac{\sqrt 3}{2}, 
     \quad 
     \tilde\varphi(\delta^{+}) = \tilde\varphi(1-\delta^{+}) =
     \frac{\sqrt 3}{2}.
     \]
   \end{lemma}  
     We note that the value of \( \tilde\varphi \) at 
     the turning points  \( \delta^{\pm} \) and \( 1-\delta^{\pm} \)
 do not depend on \( k \), whereas the
     values at the turning points  \( 0, 1/2 \) 
     do depend on \( k \) and approach \( 
     \pm\infty \) as \( k \to \infty \).
	 
     \begin{proof}
\( \tilde\varphi \) has no zeros since \( \mathcal P_{2} \) has no
real solutions. 	 
    Asymptotes  are 
    solutions of \( \mathcal P_{2}'(\psi_{c})=0 \) 
    which amounts to \(  \tilde\psi(\tilde y) = -{1}/{2} \) or 
       \(
       \tilde y=\frac{1}{2}\cos^{-1}{-\frac{1}{4\pi k}}
       \),  which correspond precisely to \( \tilde y = \delta^{*} \)
       and \( \tilde y = 1-\delta^{*} \). 
   To compute the turning points of \( \tilde\varphi \) we
   differentiate to get 
       \begin{equation}\label{derivtildephi}
       \tilde\varphi'=\frac{d\tilde\varphi}{d\tilde y} 
       =
       \frac{d\tilde\varphi}{d\tilde\psi_{c}}\frac{d\tilde\psi_{c}}{d\tilde
       y}
   = \frac{2(\mathcal P_{2}'^{2}-\mathcal P_{2}\mathcal
   P_{2}')}{\mathcal P_{2}'^{2}}   \frac{d\tilde\psi_{c}}{d\tilde
       y}  = \frac{2\mathcal P_{1}(\tilde\psi_{c})}{\mathcal
       P_{2}'^{2}(\tilde\psi_{c})}
       \tilde\psi_{c}'
   \end{equation}
 where   
 \( \mathcal P_{2}'^{2}-\mathcal P_{2}\mathcal P_{2}'' 
 = (2\tilde\psi_{c}+1)^{2}- 2(\tilde\psi_{c}^{2}+\tilde\psi_{c}+1) 
 = 2\tilde\psi_{c}^{2}+2\tilde\psi_{c}-1 = \mathcal
 P_{1}(\tilde\psi_{c})\)
%
and 
       \(
       \tilde\psi'(\tilde y) =(2\pi)^{2}k\sin 
       (2\pi \tilde y) . 
       \)
       Therefore \( \tilde\varphi'=0 \)  if 
       \(
       2\psi^2+2\psi-1 = 0
       \) 
       which gives \(  \psi(\tilde y) =  -(1\pm\sqrt{3})/2 \) and thus
       \(
       \tilde y= \{\delta^{-}, \delta^{+}, 1-\delta^{+}, 
       1-\delta^{-}\}
       \) 
       or \( \tilde\psi'=0 \) which gives 
       \(
       \tilde y = \sin^{-1} 0 = \{0, 1/2\}.
       \)
   To compute the images of these turning points,
   we have \( \tilde\psi(0) = 2\pi k \) and
   \( \tilde\psi(1/2) = 12\pi k \)
   and so 
   \(\tilde\varphi (0) \approx -k \) 
   and
   \( \tilde\varphi (1/2) \approx -k. \)
  \enlargethispage{1cm}
  By the
   definition of \( \tilde\psi  \)  we can compute  
   \(
   \tilde\psi (\delta^{-} ) = 2\pi k \cos \left( 2\pi
   \left(\tfrac{1}{2\pi}\cos^{-1}((-1+\sqrt 3)/(4\pi k))\right)\right) =
   (-1+\sqrt 3)/2
   \)
   and also
   \(
   \tilde\psi (\delta^{+} ) = 2\pi k \cos \left( 2\pi
   \left(\frac{1}{2\pi}\cos^{-1}(-1-\sqrt 3)/(4\pi k)\right)\right) =
   -(1+\sqrt 3)/(2)
   \). 
  Therefore so 
   \( 
   \tilde\varphi(\delta^{-}) 
   =-(\tilde\psi^{2}+\tilde\psi+1)/(1+2\tilde\psi)
   = -\sqrt 3/2 
   \) 
and 
\( 
   \tilde\varphi(\delta^{+}) = \sqrt 3/2 
   \) as required.
   \end{proof}

\subsection{First order hyperbolic coordinates in backward time}
We are now ready to compute hyperbolic
coordinates in backward time. First of all we define 
\begin{equation}\label{thetaminus1}
 \theta^{(-1)}(\tilde y) = 
    \begin{cases}
\frac{\pi}{2} + \frac{1}{2}\tan^{-1} (\tilde\varphi (\tilde y)) &\text{ if } 
\tilde y\in [0, \delta^{*}]\cup [1-\delta^{*}, 1]\\
\frac{1}{2}\tan^{-1} (\tilde\varphi (\tilde y)) &\text{ if } 
\tilde y  \in [\delta^{*}, 1-\delta^{*}]
\end{cases}
\end{equation}
Then we have the following

\begin{proposition}\label{e1-}
    For every \( k>0 \), hyperbolic coordinates for \( f_{k}^{-1} \)
    are defined at every point \( (x,y)\in \mathbb T^{2} \) and depend
    only on \( \tilde y\). The vector 
    \[
    e^{(-1)}(z)  = (\sin\theta^{(-1)}(\tilde y),
    \cos \theta^{-1}(\tilde y))
    \]
    is a unit vector in the most contracted direction for \( Df_{k}^{-1} \)
    at \( (x,y) \). 
\end{proposition} 
    
A graph of \( \theta^{(-1)}(\tilde y) \) is shown in Figure \ref{theta1back}
below. 
\begin{figure}[h]
    \begin{center}
	\includegraphics[width=16cm, height=4cm]{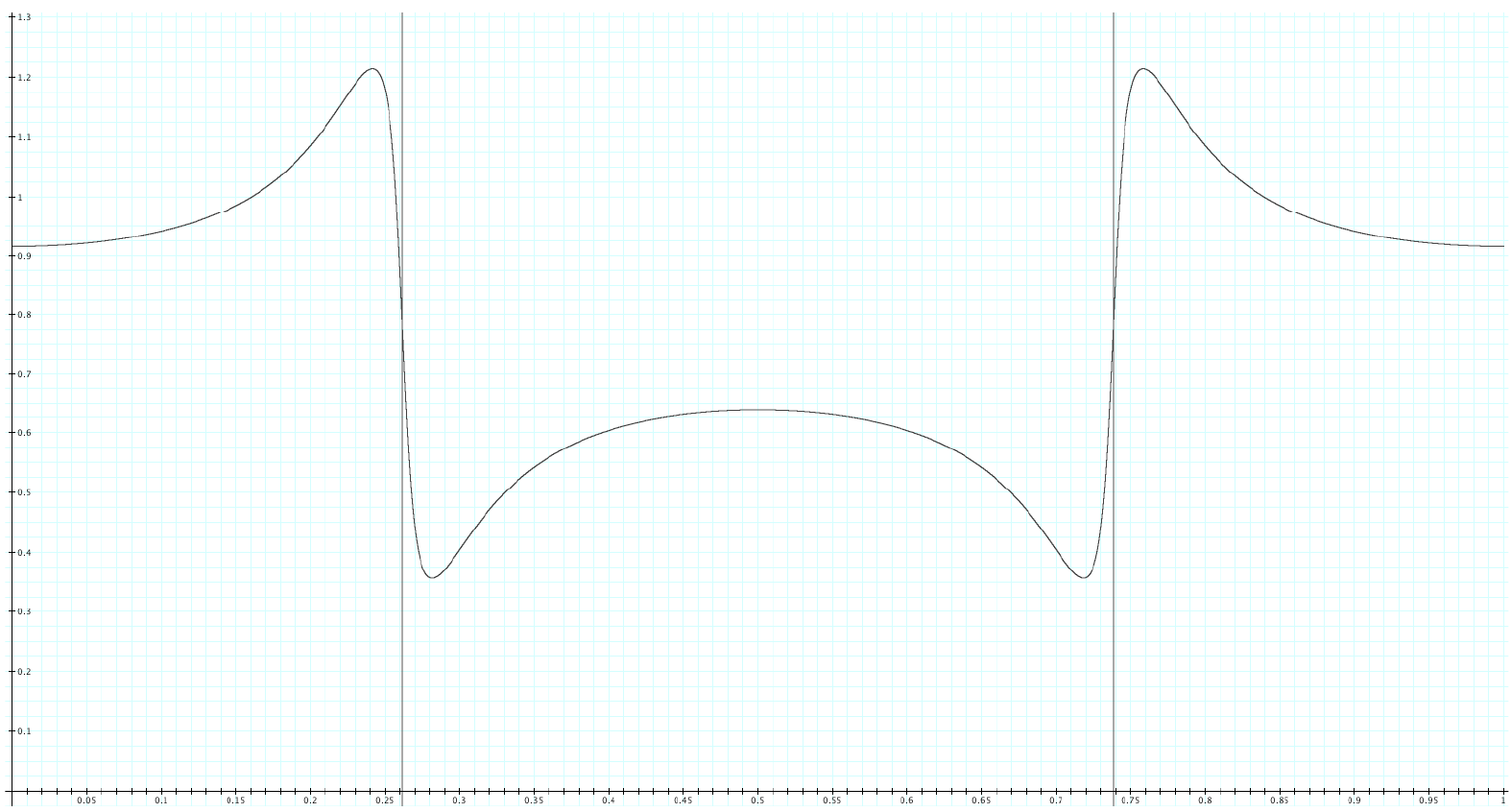}
\end{center}
\caption{Graph of \( \theta^{(-1)}(\tilde y) \).}
\label{theta1back}
\end{figure}

\begin{corollary}\label{e1-cor}
  \( e^{(-1)} \) rotates \textbf{counter-clockwise} in \( (0,
  \delta^{-1}) \) from \( \theta^{(-1)}(0)\gtrsim \pi/4 \)
  (just above the positive diagonal) for \( \tilde 
  y = 0 \), with \( \theta^{(-1)}(0) \searrow \pi/4 \)  as \( k \to
  \infty \), to an angle \( \theta^{(-1)}(\delta^{-}) \gtrsim 3\pi/8
  \), independent of  \( k \), then \textbf{swings rapidly} 
  \textbf{clockwise} in \( (\delta^{-}, \delta^{+}) \) to an angle \( 
  \theta^{(-1)}(\delta^{+}) \lesssim \pi/8 \), independent of \( k
  \), crossing the positive diagonal at \( \tilde y = \delta^{*} \), 
  then rotates \textbf{counter-clockwise}   in \( (\delta^{+}, 1/2) \)
  to an angle \( \theta^{(-1)}(1/2) \lesssim \pi/4 \) (just below the 
  positive diagonal)  
  with \( \theta^{(-1)}(1/2) \nearrow \pi/4 \)  as \( k \to
    \infty \). The process continues in \( (1/2, 1) \) in reverse in 
    a perfectly symmetric manner.
   
\end{corollary}  

\begin{proof}[Proof of Proposition \ref{e1-}]
 The argument follows along the same lines as the proof of Proposition
 \ref{e1+}. The
 existence of hyperbolic coordinates everywhere follows directly from 
 the analogous statement in forward time since \( e^{(-1)}(z) \) is
 exactly the image of the most expanded direction \(
 f^{(1)}(f^{-1}(z)) \) under \( Df_{f^{-1}(z)} \). Then, substituting 
 the corresponding partial derivatives into \eqref{contrdir} we get 
 \[ 
 \tan 2 \theta = \frac{2(-\tilde \psi_{c}(1+\tilde\psi_{c})
 -1)}{(1+\tilde\psi_{c})^{2}+1-\tilde\psi_{c}^{2}-1}=
 -\frac{2(\tilde\psi_{c}^{2} + \tilde\psi_{c}+1)}{2\tilde\psi_{c}+1}
 =\tilde\varphi(\tilde y).
 \]
 For \( \tilde y = 0 \)  we have \( \tilde \psi_{c} = 2\pi k \cos 2
 \pi 0 = 2\pi k \) and so this gives 
 \(
 \tan 2 \theta = -2(4\pi^{2}k^{2}+2\pi k + 1)/(4\pi k + 1)
 \to -\infty \) as \( k\to \infty \) and so, taking \( \tan^{-1} \) on
both sides, this gives  \( \theta \gtrsim -\pi/4 \)
 with \( \theta \searrow - \pi/4 \) as \( k \to \infty \). Thus, for
 large \( k \), the hyperbolic coordinates are almost aligned with the
 diagonals. To establish which direction is contacted and which
 direction expanded we note first of all that for 
 \( \tilde\psi_{c}(0) = 2\pi k \) and therefore, 
 applying the derivative at \( \tilde y =0 \) to a 
 vector which is close to the negative diagonal we get 
 \[ 
 Df_{\tilde y=0}^{-1}(1\pm \varepsilon, -1) 
 = \begin{pmatrix}
 1+\tilde\psi_{c} & -\tilde\psi_{c} \\ -1 & 1
 \end{pmatrix}
 \begin{pmatrix} 1 \pm \varepsilon \\ -1\end{pmatrix}
     =
     \begin{pmatrix}
      1+2\pi k & -2\pi k \\ -1 & 1
      \end{pmatrix}
      \begin{pmatrix} 1 \pm \varepsilon \\ -1\end{pmatrix}   
\approx   \begin{pmatrix} 4\pi k \\ - 2\end{pmatrix} 
 \]
The region around the negative diagonal is therefore 
clearly not the contracting direction, and thus we 
define \( \theta^{(-1)}(0) = \pi/2 + \tan^{-1}\tilde\varphi (0) \).
Since the most contracting direction depends smoothly on \( \tilde y \)
this definition continues to work as long as the hyperbolic
coordinates do not swing ``through'' the diagonals, which happens
exactly when \( \tan 2 \theta = \pm \infty \) or \( 2\tilde\psi_{c}+1 
= 0 \), i.e. for \( \tilde y = \delta^{*} \) and \( \tilde y = 1-
\delta^{*} \).  Between these two points the formula picks up the
contracting direction, which now lies in the correct quadrant, and we 
obtain the statement with the definition of \( \theta^{(-1)} \) as
given above. 
 \end{proof}   

 \subsubsection{Geometry of the most contracting directions in
 backward time}  
 As in the forward time case, we now carry out a more detailed
 analysis of the function \( e^{(-1)} \) to determine the variation
 of \( \theta^{(-1)} \) on \( \tilde y \). 
 
\begin{proof}[Proof of Corollary \ref{e1-cor}]
 To find the turning points of \( \theta^{(-1)} \) we calculate the
 derivative. From \eqref{derivtildephi} we have   
     \[ 
 \frac{d\theta^{(-1)}}{dy} =
 \frac{1}{2}\frac{\tilde\varphi'}{1+\tilde\varphi^{2}}   
 = \frac{\mathcal P_{1}(\tilde\psi_{c}) \tilde\psi_{c}'}
 {\mathcal P_{2}'^{2}(1+\tilde\varphi^{2})}
     \]
  which immediately gives the six turning points at \( 0, 1/2 \) from \( 
  \tilde\psi_{c}=0 \) and at \( \delta^{-}, \delta^{+}, 1-\delta^{-} \)
  and \( 1-\delta^{+} \) from \( \mathcal P_{1}=0 \). We compute the
  value of \( \theta^{(-1)} \) at these turning points. 
At \( \tilde y = 0 \) we have \( \tilde\psi(0) = 2\pi k \) and
\( \tilde\varphi (0) \approx -k  \) and therefore, 
from \eqref{thetaminus1},
\[ \theta^{(-1)}(0)=\frac{\pi}{2}+\frac{1}{2}\tan^{-1}(\tilde\varphi(0))
 \gtrsim \frac{\pi}{2}+\frac{1}{2}\tan^{-1}(-k) \gtrsim \frac{\pi}{4}.
 \]
 Thus \( e^{(-1)} \) lies just slightly above the positive diagonal,
 tending to the positive diagonal as \( k \) increases. 
 For \( \tilde y \) between \( 0 \) and \( \tilde y = \delta^{-}\),
 \( \tilde\varphi(\tilde y) \) increases and 
 \( e^{(-1)} \) rotates \emph{counterclockwise}. At 
 \( \tilde y = \delta^{-} \) (and \( \tilde y = 1-\delta^{-} \)) we have 
 (see Lemma \ref{l:tildevarphi}) 
 \(
 \tilde\varphi(\delta^{-}) =
 -\sqrt 3/2 \) and so 
 \[
 \theta^{(-1)} = \frac{\pi}{2}+\frac{1}{2}\tan^{-1}\left(-\frac{\sqrt
 3}{2}\right)
 \gtrapprox \frac{\pi}{2} - \frac{\pi}{8} = \frac{3\pi}{8}
 \]
 We emphasize that this value is independent of \( k \). 
 Past \( \delta^{-} \), \( \tilde\varphi(
 \tilde y) \) decreases and \( e^{(-1)} \) rotates clockwise. At \(
 \tilde y = \delta^{+} \) (and \( \tilde y = 1- \delta^{+} \)) we have 
\( \tilde\varphi (\delta^{+}) = \sqrt 3/2 \) and so 
\[ \theta^{(-1)} = \frac{1}{2}\tan^{-1}\frac{\sqrt 3}{2}
\lessapprox \frac{\pi}{8}. \]
 This value is also independent of \( k \) so the  vector \( e^{(-1)} \)
 swings through a \emph{fixed} angle of almost \( \pi/4 \) as \(
 \tilde y \) ranges in  the 
 interval \( [\delta^{-}, \delta^{+}] \) whose length of the order \(
 1/k \) and which \emph{shrinks to 0} as \( k \to 
 \infty \).  Finally, the direction of \( e^{(-1)} \) changes again
 and now rotates in a counter-clockwise sense. At \( \tilde y = 1/2 \)
 we have \( \tilde\varphi \approx -k \) and so 
 \[ 
 \theta^{(-1)}(1/2)=\frac{1}{2}\tan^{-1}(\tilde\varphi(1/2)) 
 \lessapprox \frac{1}{2}\tan^{-1}k \lessapprox \frac{\pi}{4}.
 \] 
\end{proof}

\subsection{First order foliations in backwards time}

The geometry of the stable and unstable foliations \( \mathcal E^{(-1)} \)
and \( \mathcal F^{(-1)} \) now follows by a careful consideration of
the position of the vectors \( e^{(-1)} \) as described in Corollary
\ref{e1-cor}.

\section{Curves of tangencies}

In this section we prove Theorem \ref{th:tang}. 
Recall first of all that 
\( \mathcal C^{(1)} \)  is, by definition, the locus of points 
\( z \)  for which 
\begin{equation}\label{e1=e-1}
e^{(1)}(z) = e^{(-1)}(z). 
\end{equation}
Our strategy is to obtain an explicit formula for this set by equating the
explicit formulas for \( e^{(1)}(z) \) and \( e^{(-1)}(z) \), although
we emphasize that 
this does not follow simply by substituting the corresponding 
expressions into \eqref{e1=e-1} as they are piecewise defined in various 
regions of the torus and also \( e^{(1)} \) is given as a function of \( 
y \) whereas \( e^{(-1)} \) is given as a function of \( \tilde y \); 
some non-trivial analysis is therefore required. We shall continue to 
use the definitions and notation of the previous section, although we 
recall some of the definitions here for clarity.  
First of all, writing \( e^{(1)}=(\cos \theta^{(1)}, \sin \theta^{(1)}) \) and 
\( e^{(-1)}=(\cos \theta^{(-1)}, \sin \theta^{(-1)}) \) the equation \( 
e^{(1)}(z) = e^{(-1)}(z) \) reduces to the equation 
\begin{equation}\label{equalangles}
\theta^{(1)}(z) = \theta^{(-1)}(z).
\end{equation}
We start by 
dividing the torus \( \mathbb T^{2} \) into two regions. 
In one region we show directly that no tangencies can occur, in the
other we show that they can be computed through a simplified formula.

\begin{lemma}
There are no tangencies
in \( \{(x,y)\in \mathbb T^{2}: y\in [0,
\delta^{-}]\cup[1-\delta^{-},1] \}\).
\end{lemma}
\begin{proof}
    From the equations defining the contracting directions \( e^{(1)}\)
   and  \( e^{(-1)} \) it follows that \( e^{(-1)} \) \emph{always}
   lies in the \( \{x>0, y>0\} \) quadrant  whereas, for 
   \( y \in [0, \delta^{-}]\cup[1-\delta^{-},1]  \) the vector \(
   e^{(-1)} \) lies in the quadrant \( \{x<0, y>0\} \). Thus they can 
   never be aligned and there can be no tangencies in this region. 
\end{proof}     

\begin{lemma}
Tangencies in 
\( \{(x,y): y\in
[\delta^{-}, 1-\delta^{-}]\} \) are given by
solutions to 
\( \varphi(y) = \tilde\varphi (\tilde y) \).  
\end{lemma}

\begin{proof}
We divide the region \( \{(x,y): y\in
[\delta^{-}, 1-\delta^{-}]\} \) into two parts as follows: 
\begin{align*}
\mathcal R_{1} &= \{(x,y) : y\in [\delta^{-}, \delta^{+}] \cup 
[1-\delta^{+}, 1-\delta^{-}], \tilde y\in [0,
\delta^{*}]\cup [1-\delta^{*}, 1]\}, 
\\ 
\mathcal R_{2} &= \{(x,y): y\in (\delta^{+}, 1-\delta^{+}],    \tilde y \in 
[\delta^{*}, 1-\delta^{*}]\}
\end{align*}
It then just remains to notice that these the definitions of \(
\theta^{(1)} \) and \( \theta^{(-1)} \) give \(
\theta^{(1)}=\theta^{(-1)} \) if and only if \( \varphi(y) =
\tilde\varphi(\tilde y) \) precisely when \( (x,y) \in \mathcal
R_{1}\cup \mathcal R_{2}\).
\end{proof}

\begin{proof}[Proof of Theorem \ref{th:tang}]
From the previous discussion it follows that we just need to 
solve the equation
\begin{equation}\label{equalphi}
\varphi(y) = \tilde\varphi (\tilde y).
\end{equation} 
Figure  \ref{sidebyside} shows the graphs of 
these two functions side by side.
\begin{figure}[h]
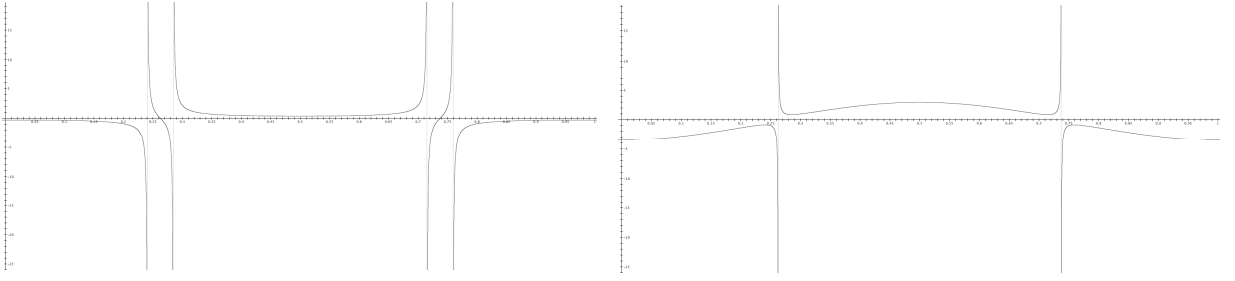

    \begin{center}
\includegraphics[width=8cm]{contrdirgraphf}
    \includegraphics[width=8cm]{contrdirgraphfinv}
    \caption{Graph of \( \varphi(y) \) and \( \tilde\varphi (\tilde y) \).}
    \label{sidebyside}
 \end{center}   
  \end{figure} 
  Notice that we need to consider the two regions \( \mathcal R_{1} \)
  and \( \mathcal R_{2} \) separately. These regions are easily
  defined in both graphs by the asymptotes which occur at \(
  \delta^{-}, \delta^{+}, 1=\delta^{+}, 1-\delta^{-} \) for \( \varphi \)
  and \( \delta^{*}, 1-\delta^{*} \) for \( \tilde\varphi \). 
  Equation \eqref{equalphi} can be written as
  \[
  y = \varphi^{-1}(\tilde\varphi (\tilde y)
  \]
  where the appropriate inverse branch is considered depending on the
  value of \( \tilde y \) so that we have 
  \begin{equation}\label{tangencyeq}
  y= \begin{cases}
  \varphi^{-1}(\tilde\varphi (\tilde y)) \cap 
  [\delta^{-}, \delta^{+}] \cup 
  [1-\delta^{+}, 1-\delta^{-}] & \text{ if } \tilde y\in [0,
\delta^{*}]\cup [1-\delta^{*}, 1]\\
\varphi^{-1}(\tilde\varphi (\tilde y)) \cap 
[\delta^{+}, 1-\delta^{+}] & \text{ if } 
\tilde y \in 
[\delta^{*}, 1-\delta^{*}].
\end{cases}
  \end{equation}
  Notice that in both cases we get exactly two solutions 
  corresponding to the two curves of tangencies.  
To calculate a more explicit expression for this we compute \( \varphi^{-1} \)
explicitly. Writing 
\[
\varphi(y)=-\frac{4\psi_{c} +2}{2\psi_{c}^2+2\psi_{c}-1} = z
 \]
 This gives the quadratic equation 
 \( 2z\psi_{c}^2 +2(z+2)\psi_{c}
    -z+2=0 \) in \( \psi_{c} \).
Solving for \( \psi_{c} \) this gives
 \[
 \psi_{c}=
 \frac{-2(z+2)\pm\sqrt{4(z+2)^2+
 8z(z-2)}}{4z}
 =
 \frac{-(z+2)\pm\sqrt{3z^2 + 4}}{2z}
 \]
From the
definition of \( \psi_{c}=\psi_{c}(y) = 2\pi k \cos (2\pi y) \) we
then get 
\begin{equation}\label{varphiinv}
y=\varphi^{-1}(z) 
=\frac{1}{2\pi}\cos^{-1}\frac{-(z+2)\pm\sqrt{3z^2+
 4}}{4\pi k z}
\end{equation}  

Notice that the expression under the square root is always positive.
The last thing to check is that the curves are contained in the two
thin strips as stated in the Theorem. To show this 
we  now define the values of 
\( \hat\delta_{T}^{-}, \delta_{T}^{-},    \delta_{T}^{+}, \hat \delta_{T}^{+}  \)
as 
\begin{align*}
    &\hat\delta_{T}^{-} :=
 \varphi^{-1}(\tilde\varphi (\delta^{-})) \cap [\delta^{-},
  \delta^{+}]  
\\ &
  \hat\delta_{T}^{+} := \varphi^{-1}(\tilde\varphi (\delta^{+})) \cap [\delta^{+},
  1/2] 
\\ & \delta_{T}^{-} := \varphi^{-1}(\tilde \varphi (0)) \cap 
  [\delta^{-},    \delta^{+}]  
  \\ & 
  \delta_{T}^{+}:= \varphi^{-1}(\tilde\varphi (1/2)) \cap 
   [\delta^{+}, 1/2]. 
\end{align*}
\begin{figure}[h]
    \begin{center}
  \includegraphics[width=5cm]{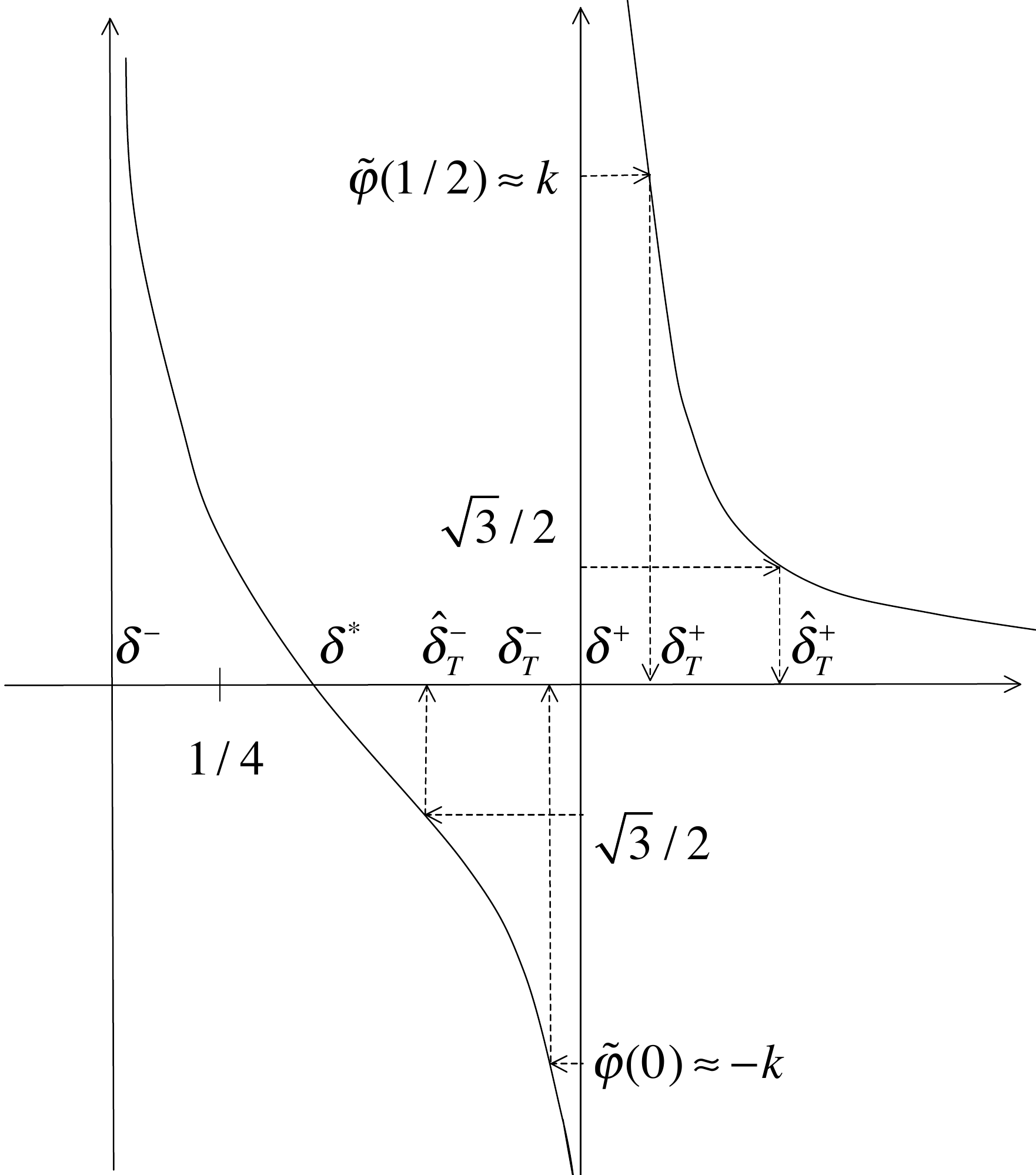}
\end{center}
\caption{}
  \label{varphigraph}
\end{figure}
It then follows, just by the geometry of the
  graphs of \( \varphi \) and \( \tilde\varphi \), see Figure
  \ref{varphigraph}, that 
\[
\delta^{-}< 1/4 < \delta^{*} < \hat\delta_{T}^{-} < \delta_{T}^{-} < 
    \delta^{+} < \delta_{T}^{+}< \hat \delta_{T}^{+} 
 \]
 and that all these
 constants tend to \( 1/4 \) (and thus to each other) as \( k\to \infty \)
 (although we shall not use this in any significant way, notice
 nonetheless that \( |\delta^{+}_{T}-\delta^{-}_{T}| \approx 1/k^{2} \)
 whereas \( |\hat \delta^{+}_{T}-\hat \delta^{-}_{T}| \approx 1/k  \)).
The statement then 
follows from an analysis of the graph in Figure \ref{sidebyside},
the formula \eqref{tangencyeq},
and the definitions made above. Indeed, notice that 
    as \( \tilde y \) ranges over the interval \( [0, \delta^{*}] \) 
    the lowest value attained by \(
    y=\varphi^{-1}(\tilde\varphi(\tilde y)) \) restricted to \(
    [\delta^{-}, \delta^{+}] \) is precisely \( \hat\delta^{-}_{T}\)
    which is attained at the point \( \tilde y = \delta^{-} \) where \( 
    \tilde\varphi \) has a local maximum. Similarly, as \( \tilde y \)
    ranges over the interval \( [\delta^{*}, 1/2] \), the highest
    value attained by 
    \(y=\varphi^{-1}(\tilde\varphi(\tilde y)) \) restricted to 
    \( [\delta^{+}, 1/2] \) is precisely \( \hat\delta^{+}_{T} \) 
    which is attained at the point \( \tilde y = \delta^{+} \) which
    is where \( \tilde\varphi \) has a local maximum. Exactly the same
    analysis works for the other regions. 
    This completes the proof of Theorem \ref{th:tang}.
\end{proof}

Before ending this section, we use the formula we have obtained to
compute a few explicit points on the curve of tangency to obtain
Figure \ref{tangencygraph}. 
From the definitions above we immediately have an explicitly defined
set of tangencies which includes 
the following points which we express in 
\( (\tilde y, y) \) coordinates: 

\begin{center}
\begin{tabular}{c | c ccccccc}
    \hline
    \(  \tilde y \) &  \( 0 \) & \( \delta^{-} \) & \( \delta^{*} \) & \(
     \delta^{+} \) & \( 1/2 \) & \( 1-\delta^{+} \) 
     & \( 1-\delta^{*} \) & \( 
     1-\delta^{-} \)  \\
    \( y \) & \( \delta^{-}_{T} \) & \( \hat\delta^{-}_{T} \) & \( \delta^{+}  \)
    & \(  \hat\delta^{+}_{T}  \) & \( \delta^{+}_{T}  \) & 
    \(  \hat\delta^{+}_{T}  \) &  \( \delta^{+} \) & 
    \(  \hat\delta^{-}_{T} \) \\
    \hline
 \end{tabular}
\end{center}
Notice that for
large \( k \) the argument of \( \cos^{-1} \) is close to \( 0 \).
Thus, using the fact that \( \cos'(0)=-1 \) we immediately have the
following limits as \( k\to \infty \): 
\[ \delta^{-}(k) \to \frac{1}{4}-\frac{\sqrt 3-1}{8\pi^{2}k}; 
\delta^{*}(k) \to \frac{1}{4}+\frac{1}{8\pi^{2}k};
\delta^{+}(k) \to \frac{1}{4}+\frac{1+\sqrt 3}{8\pi^{2}k};
\]
\[
\hat\delta^{-}_{T} \to \frac{1}{4}+\frac{1+\sqrt 3/3}{8\pi^{2}k};
\hat\delta^{+}_{T} \to \frac{1}{4}+\frac{1+3\sqrt 3}{8\pi^{2}k}.
\]
In particular this allows us to study the positions of various points.
We have 
\[ \delta^{+}-\delta^{*} \to \frac{\sqrt 3}{8\pi^{2}k};
\quad 
\hat\delta^{-}_{T}-\delta^{-} \to \frac{\frac{4}{3}\sqrt 3}{8\pi^{2}k} 
\quad
\hat\delta^{+}_{T}-\delta^{+} \to \frac{2\sqrt 3}{8\pi^{2}k}; 
\]
This implies the relative positions as illustrated in Figure
\ref{tangencygraph}.

\section{Uniform hyperbolicity}

In this section we prove Theorem \ref{th:hyp}. We divide the proof
into two parts, estimating the direction of \( \tilde v \) and its
magnitude respectively. 
First of all, however, we mention a couple of elementary consequences 
of our assumptions on the location of \( z \) and the direction of \( 
v \). 
From the definition of \( \psi_{c} \) we have 
\[ \psi_{c}(\delta^{(\pm m)}) = 2\pi k 
\cos \left(2\pi \left(\frac{1}{2\pi}\cos^{-1}\frac{\pm m}{\pi 
k}\right) \right)  = \pm 2m \]
and therefore, \( z\notin\Delta^{(m)} \) implies 
\[
|\psi_{c}|>2m.
\]
Also, from our assumptions \( \tan\theta\in (m^{-1}, m) \) for the
direction of \( v \) and \( m\geq 2 \) we  get 
\[
m^{-1}\sin\theta < \cos\theta < m \sin\theta
\quad\text{ and } \quad \sin\theta > m^{-1}/2.
\]

\subsection{Direction}

To estimate the direction of the image vector \( \tilde v = Df_{z}(v) 
= (\cos \tilde\theta, \sin\tilde\theta) \) we recall that 
the derivative of the standard map \( f=f_{k} \) at a point \( z=
(x,y) \) depends only on \( y \) and is 
\[ Df_y= \begin{pmatrix} 1 & 2\pi k \cos (2\pi y) \\ 1 & 1 + 2\pi k
\cos (2\pi y) \end{pmatrix} = 
\begin{pmatrix} 1 & \psi_{c}\\ 1 & 1 + \psi_{c} \end{pmatrix}
\]
Thus the image \( \tilde v = (\cos\tilde\theta, \sin\tilde\theta) \) 
of a generic vector \( v= ( \cos \theta, \sin \theta) \)
is 
\[ \begin{pmatrix} \cos \tilde \theta \\ \sin\tilde \theta\end{pmatrix} 
= Df_{y}\begin{pmatrix} \cos \theta \\ \sin\theta\end{pmatrix} 
=   \begin{pmatrix} 1 & \psi_{c}\\ 1 & 1 + \psi_{c} \end{pmatrix}
    \begin{pmatrix} \cos \theta \\ \sin\theta\end{pmatrix} 
=\begin{pmatrix} \cos\theta + \psi_{c} \sin\theta 
\\ \cos\theta + ( 1 + \psi_{c}) \sin\theta \end{pmatrix}
\]
Therefore 
\begin{equation}\label{thetatilde}
    \tilde\theta = \tan^{-1}\left( 
\frac{\cos\theta + ( 1 + \psi_{c}) \sin\theta }{\cos\theta + \psi_{c} \sin\theta}
\right)
= 
\tan^{-1}\left( 
1+ \frac{\sin\theta }{\cos\theta + \psi_{c} \sin\theta}
\right)
\end{equation}
Therefore it is sufficient to bound the absolute value of the 
fraction \(  {\sin\theta }/{(\cos\theta + \psi_{c} \sin\theta)}\) by \( 
m^{-1} \). 
From the condition \( \tan\theta \in(m^{-1}, m) \) we know
 that \( \sin\theta \) and \( \cos\theta \) have the same sign. 
 Suppose without loss of generality that \( \sin\theta > 0 \) and
 \( \cos\theta > 0 \).    We consider two different cases,
 corresponding to different regions of \( \mathbb
 T^{2}\setminus\Delta^{(m)} \). If \( \psi=\psi_{c}(y) > 0 \)  we have 
 \[ 
 \frac{|\sin\theta|}{|\cos\theta + \psi\sin\theta|} 
 =  \frac{\sin\theta}{\cos\theta + \psi\sin\theta} 
 \leq  \frac{\sin\theta}{\psi_{c}\sin\theta} = \frac{1}{\psi} 
 \leq \frac{1}{2m} <  m^{-1}.
 \]
 If \( \psi = \psi_{c}(y) < 0 \) we distinguish two further subcases. 
 If \( \cos\theta - |\psi|\sin\theta \geq   0 \) 
 then 
 \(
 \cos\theta -|\psi|\sin\theta >
 m^{-1}\sin\theta - |\psi|\sin\theta = (m^{-1}-|\psi|) \sin\theta 
 \)
 and
 therefore 
 \[ \frac{|\sin\theta|}{|\cos\theta + \psi\sin\theta|} =
 \frac{\sin\theta}{\cos\theta - |\psi|\sin\theta} 
 \leq \frac{\sin\theta}{(m^{-1}- |\psi|) \sin\theta} 
 = \frac{1}{m^{-1} - |\psi|}  < \frac{1}{m}.
 \]
 If \( \cos\theta - |\psi|\sin\theta \leq 0 \) then 
 \(
 \cos\theta - |\psi|\sin\theta \leq  m\sin\theta -
 |\psi|\sin\theta = (m-|\psi |) \sin\theta
 \)
 and so 
 \(
 |\cos\theta - |\psi\sin\theta| \geq  (m-|\psi|) \sin\theta
 \)
 and therefore 
 \[ \frac{|\sin\theta|}{|\cos\theta + \psi\sin\theta|} =
 \frac{\sin\theta}{|\cos\theta - |\psi|\sin\theta|} 
 \leq \frac{\sin\theta}{|m - |\psi|| \sin\theta} 
 = \frac{1}{|m- |\psi||} \leq \frac{1}{m}.
\]
 This completes the proof of the invariance of the conefield.

\begin{remark}
 Note that the formula \eqref{thetatilde} describes a \emph{general} relation
 between the angle \( \theta \) of a vector \( v  \), the point \( z
 \), and the angle \( \tilde\theta \) of the image vector \( \tilde v 
 = Df_{z}(v) \) (in particular it holds for \emph{every} \( \theta \) and
 \emph{every} \( z \)). We mention here two simply but possibly interesting 
 consequences of this
 formula even though they do not have a direct application for our
 immediate purposes.  First of all, 
 for any \( y \),  if \( \sin\theta = 0 \) (horizontal) then 
\( \tilde\theta = \tan^{-1}1 \) (positive diagonal). Thus the
\emph{horizontal vectors are always mapped to the positive diagonal.} 
Secondly,  notice that 
\( \psi_{c}= \psi_{c}(y)  = 2\pi k \cos (2\pi y)  \) is the only term in 
which the location of \( z=(x,y)  \)  comes into the equation. In
particular, for \( y= 1/4 \) and \( y=3/4 \) we have \( \psi_{c}=0 \) 
and this equation reduces to 
\[ \tilde\theta = 
\tan^{-1}\left( 
1+ \frac{\sin\theta }{\cos\theta}
\right) = \tan^{-1}(1+\tan \theta).
\]
Therefore, if \( \tan\theta = -1 \) (negative diagonal) we have \(
\tilde\theta = 0 \), (horizontal). Thus, 
\emph{when \( y= 1/4 \) and \( y=3/4 \) the negative diagonal is mapped to the
horizontal.} 
\end{remark}

 \subsection{Expansion}
 
We now want to estimate the size of the vector \( \tilde v = Df_{z}(v)
\).  By simply calculating the norm we get 
 \begin{align*}
 \|\tilde v\| & 
 = \sqrt{(\cos\theta+\psi\sin\theta)^{2}+ 
 (\cos\theta + (1+\psi) \sin\theta)^{2}}
\\ &=
\sqrt{2\cos^{2}\theta + 2(1+2\psi)\cos\theta\sin\theta + 
\psi^{2}(1+\psi)^{2} \sin^{2}\theta}
\\ & \geq 
\sqrt{4\psi \cos\theta\sin\theta + \psi^{2}(1+\psi)^{2} \sin^{2}\theta}
\\ & \geq \sqrt{4\psi m^{-1}\sin^{2}\theta+ 
\psi^{2}(1+\psi)^{2} \sin^{2}\theta } 
\\ & = \sin\theta \sqrt{4\psi m^{-1} + \psi^{2}(1+\psi)^{2}} 
\end{align*}
 Once again we distinguish two cases depending on the sign of \( \psi \).
 If \( \psi > 0 \) the estimate is straightforward and we get, using \( 
 \sin\theta > m^{-1}/2 \)  and \( \psi\geq 2m \)
 \[  \|\tilde v\|  \geq 
 \sin\theta \sqrt{\psi^{4}} 
 \geq 4m^{2}/2m = 2m > m.\]
 If \( \psi<0 \) we write 
  \[
  \psi^{2}(1+\psi)^{2}  + 4\psi m^{-1} = 
     \psi^{2}\left[(1+\psi)^{2}  + \frac{4\psi}{m \psi^{2}}  \right] 
 \geq \psi^{2}\left[(2m-1)^{2}  - \frac{2}{m^{2}}  \right] 
 \]
 Since \( m\geq 2 \) we have \( m^{2}\geq 4 \) and \( 2/m^{2}\leq 1/2 \)
 and therefore 
 \[
 (2m-1)^{2}-\frac{2}{m^{2}}\geq 4m^{2}-4m + 1 - \frac{2}{m^{2}} \geq  
 4m^{2}-4m = 4m^{2}(1-\frac{1}{m}) \geq 2m^{2}.
 \]
 This gives \(    \psi^{2}(1+\psi)^{2}  + 4\psi m^{-1} \geq
 2\psi^{2}m^{2} \geq 8m^{4} \) and therefore, substituting into the
 square root we get 
 \[ 
 \|\tilde v\| \geq \sin\theta \sqrt{4\psi m^{-1} + \psi^{2}(1+\psi)^{2}} 
 \geq \sin\theta \cdot \sqrt 8 m^{2} \geq \frac{\sqrt 8 m^{2}}{2m} >  m.
 \]
 This completes the proof of Theorem \ref{th:hyp}.

%

 \end{document}